\documentclass[11pt]{article}
\usepackage{cite}
\usepackage{subfigure}
\usepackage{graphicx, color, graphpap}
\usepackage{amsmath,amssymb,amsthm}
\usepackage{ifthen}
\usepackage{mathrsfs}
\usepackage{mathtools}
\usepackage{enumitem}
\usepackage{physics}
\usepackage{pstricks}
\usepackage{float}
\usepackage{algorithmic}
\usepackage{fancyhdr}
\usepackage[us,12hr]{datetime} 
\usepackage{exscale}
\usepackage{tabularx}
\usepackage{latexsym}
\usepackage{fullpage}       
\usepackage{float}
\usepackage{verbatim}
\usepackage{multirow}
\usepackage{overpic}

\usepackage{booktabs}
\usepackage{makeidx}
\usepackage{algorithm,algorithmic}
 
\numberwithin{equation}{section}  % enables section numbering
\numberwithin{table}{section}
\numberwithin{algorithm}{section}
\makeindex

\usepackage{csvsimple}

\usepackage[pagebackref=true]{hyperref}
\hypersetup{
	plainpages=false,       % needed if Roman numbers in frontpages
	unicode=false,          % non-Latin characters in Acrobat bookmarks
	pdftoolbar=true,        % show Acrobat toolbar?
	pdfmenubar=true,        % show Acrobat menu?
	pdffitwindow=false,     % window fit to page when opened
	pdfstartview={FitH},    % fits the width of the page to the window
	pdftitle={newRPRSMforQAP: 
       Restricted Dual PRSM for the QAP
	           },    % title: 
	pdfnewwindow=true,      % links in new window
	colorlinks=true,        % false: boxed links; true: colored links
	linkcolor=blue,         % color of internal links
	citecolor=green,        % color of links to bibliography
	filecolor=magenta,      % color of file links
	urlcolor=cyan           % color of external links
}
\newcommand*\mathinhead[2]{\texorpdfstring{$\boldsymbol{#1}$}{#2}}
\usepackage[noabbrev]{cleveref}
\DeclareMathOperator{\argmin}{argmin}

\usepackage{mathtools}
\DeclarePairedDelimiter{\ceil}{\lceil}{\rceil}

\usepackage{refcount} % for counting eqtn number in enumerate enviornment

\def\R{\mathbb{R}}
\def\Sc{\mathbb{S}}

\def\Snt{\Sc^{n^2+1}}

\def\Snp{\Sc_+^n}

\def\Snpp{\Sc_{++}^n}

\def\Rn{\mathbb{R}^n}

\def\bY{\overline{Y\strut}}
\def\hY{\widehat{Y}}
\def\hV{\widehat{V}}

\def\eqref#1{{\normalfont(\ref{#1})}}

\def\eqref#1{{\normalfont(\ref{#1})}}
\newtheorem{theorem}{Theorem}[section]

\newtheorem{definition}[theorem]{Definition}

\newtheorem{prop}[theorem]{Proposition}

\newtheorem{lem}[theorem]{Lemma}

\newtheorem{cor}[theorem]{Corollary}

\newtheorem{remark}[theorem]{Remark}

\newtheorem{lemma}[theorem]{Lemma}

\crefname{thm}{Theorem}{Theorems}
\Crefname{thm}{Theorem}{Theorems}
\crefname{problem}{Problem}{Theorems}
\Crefname{problem}{Problem}{Theorems}
\Crefname{assump}{Assumption}{Theorems}
\crefname{assump}{Assumption}{Theorems}
\crefname{conjecture}{Conjecture}{Theorems}
\Crefname{conjecture}{Conjecture}{Theorems}
\crefname{prop}{Proposition}{Propositions}
\Crefname{prop}{Proposition}{Propositions}
\crefname{cor}{Corollary}{Corollaries}
\Crefname{cor}{Corollary}{Corollaries}
\crefname{lem}{Lemma}{Lemmas}
\Crefname{lem}{Lemma}{Lemmas}
\theoremstyle{definition}
\crefname{defn}{definition}{definitions}
\Crefname{defn}{Definition}{Definitions}
\crefname{conj}{Conjecture}{Conjectures}
\Crefname{conj}{Conjecture}{Conjectures}
\crefname{remark}{Remark}{Remarks}
\Crefname{remark}{Remark}{Remarks}
\crefname{rmk}{Remark}{Remarks}
\Crefname{rmk}{Remark}{Remarks}
\crefname{example}{Example}{Examples}
\Crefname{example}{Example}{Examples}
\crefname{align}{}{}
\Crefname{align}{}{}
\crefname{equation}{}{}
\Crefname{equation}{}{}

\newcommand{\textdef}[1]{\textit{#1}\index{#1}}

\newcommand{\ZZ}{{\mathcal Z} }

\newcommand{\LL}{{\mathcal L} }
\newcommand{\RR}{{\mathcal R} }
\newcommand{\YY}{{\mathcal Y} }

\newcommand{\DD}{{\mathcal D} }

\newcommand{\GG}{{\mathcal G} }

\newcommand{\PP}{{\mathcal P} }

\newcommand{\whV}{{\widehat{V}} }

\newcommand{\cG}{{\mathcal G} }

\newcommand{\rPRSM}{\textbf{rPRSM}\,} % changed by TK
\newcommand{\rPRSMp}{\textbf{rPRSM}}
\newcommand{\QAP}{\textbf{QAP}\,}
\newcommand{\QAPp}{\textbf{QAP}}

\newcommand{\ADMM}{\textbf{ADMM}\,}
\newcommand{\ADMMp}{\textbf{ADMM}}

\newcommand{\PRSM}{\textbf{PRSM}\,}
\newcommand{\PRSMp}{\textbf{PRSM}}
\newcommand{\DNN}{\textbf{DNN}\,}
\newcommand{\DNNp}{\textbf{DNN}}
\newcommand{\KKT}{\textbf{KKT}\,}
\newcommand{\KKTp}{\textbf{KKTp}\,}
\newcommand{\SDP}{\textbf{SDP}\,}

\newcommand{\SDPp}{\textbf{SDP}}
\newcommand{\FR}{\textbf{FR}\,}
\newcommand{\FRp}{\textbf{FR}}

\newcommand{\NN}{{\mathcal{N}}}

\newcommand{\A}{{\mathcal A}}

\newcommand{\bbm}{\begin{bmatrix}}
\newcommand{\ebm}{\end{bmatrix}}
\newcommand{\bem}{\begin{pmatrix}}
\newcommand{\eem}{\end{pmatrix}}
\newcommand{\beq}{\begin{equation}}
\newcommand{\beqs}{\begin{equation*}}
\newcommand{\bet}{\begin{table}}
\newcommand{\eeq}{\end{equation}}
\newcommand{\eeqs}{\end{equation*}}
\newcommand{\beqr}{\begin{eqnarray}}
\renewcommand{\vec}{{\rm vec}}

\DeclareMathOperator{\Null}{null}

\DeclareMathOperator{\range}{range}

\DeclareMathOperator{\kvec}{{vec}}

\DeclareMathOperator{\diag}{{diag}}
\DeclareMathOperator{\Diag}{{Diag}}

\DeclareMathOperator{\Mat}{{Mat}}

\newcommand{\nc}{\newcommand}
\nc{\arrow}{{\rm arrow\,}}
\nc{\Arrow}{{\rm Arrow\,}}
\nc{\BoDiag}{{\rm B^0Diag\,}}
\nc{\bodiag}{{\rm b^0diag\,}}
\newcommand{\oodiag}{{\rm o^0diag\,}}

\nc{\Mm}{{\mathcal M}^{m} }
\nc{\Mmn}{{\mathcal M}^{mn} }
\nc{\Mnr}{{\mathcal M}_{nr} }
\nc{\Mnmr}{{\mathcal M}_{(n-1)r} }
\nc{\kwqqp}{Q{$^2$}P\,}
\nc{\kwqqps}{Q{$^2$}Ps}

\nc{\notinaho}{(X,S)\in \overline{AHO}(\A)}
\nc{\inaho}{(X,S)\in AHO(\A)}

\newcommand{\bea}{\begin{eqnarray}}%
\newcommand{\eea}{\end{eqnarray}}%
\newcommand{\beas}{\begin{eqnarray*}}%
\newcommand{\eeas}{\end{eqnarray*}}%
\newcommand{\Rnn}{\R^{n \times n}}%
{}
\newcommand{\Hnp}[1][]{\,\mathbb{H}_+^{\ifthenelse{\equal{#1}{}}{n}{#1}}}
\newcommand{\Hn}[1][]{\,\mathbb{H}^{\ifthenelse{\equal{#1}{}}{n}{#1}}}
\newcommand{\Hk}[1][]{\,\mathbb{H}^{\ifthenelse{\equal{#1}{}}{k}{#1}}}
\newcommand{\Dn}[1][]{\,\mathbb{D}^{\ifthenelse{\equal{#1}{}}{n}{#1}}}

\newcommand{\cP}{\mathcal{P}}

\begin{document}
\title{A Restricted Dual Peaceman-Rachford Splitting Method for \QAP}
     \author{
\href{https://uwaterloo.ca/combinatorics-and-optimization/about/people/n6graham}{Naomi
Graham }
%\thanks{
	%Department of Combinatorics and Optimization
        %Faculty of Mathematics, University of Waterloo, Waterloo, Ontario, Canada N2L 3G1; Research supported by The Natural Sciences and Engineering Research Council of Canada.
%}
\and
\href{https://huhao.org/}{Hao Hu}
%\thanks{
%	Department of Combinatorics and Optimization
%        Faculty of Mathematics, University of Waterloo, Waterloo, Ontario, Canada N2L 3G1; Research supported by The Natural Sciences and Engineering Research Council of Canada.
%}
\and
\href{https://uwaterloo.ca/combinatorics-and-optimization/about/people/n6graham}
{Haesol Im}
%\thanks{
%	Department of Combinatorics and Optimization
%        Faculty of Mathematics, University of Waterloo, Waterloo, Ontario, Canada N2L 3G1; Research supported by The Natural Sciences and Engineering Research Council of Canada.
%}
\and
\href{}{Xinxin
Li}\thanks{School of Mathematics, Jilin University, Changchun, China. 
E-mail: 
{\tt xinxinli@jlu.edu.cn}. This work was supported by the National Natural 
Science Foundation of China (No.11601183) and  Natural Science Foundation 
for Young Scientist of Jilin Province (No. 20180520212JH).}
\and
\href{http://www.math.uwaterloo.ca/~hwolkowi/}
{Henry Wolkowicz}%
%        \thanks{Department of Combinatorics and Optimization
%        Faculty of Mathematics, University of Waterloo, Waterloo, Ontario, Canada N2L 3G1; Research supported by The Natural Sciences and Engineering Research Council of Canada; 
%\url{www.math.uwaterloo.ca/\~hwolkowi}.
%}
}

%\email{jie.lin@uwaterloo.ca}
%%\affiliation{Institute for Quantum Computing and Department of Physics and Astronomy, University of Waterloo, N2L3G1 Waterloo, Ontario, Canada}

\break
\date{\currenttime, \today \\
        Department of Combinatorics and Optimization\\
        Faculty of Mathematics, University of Waterloo, Canada. \\
Research supported by NSERC.
}
\maketitle
\tableofcontents
\listoffigures
\listoftables

\begin{abstract}
We revisit and strengthen splitting methods for solving
doubly nonnegative, \DNNp, relaxations of the 
quadratic assignment problem, \QAPp. 
We use a modified restricted contractive splitting method,
\rPRSMp, approach. Our strengthened bounds and new dual multiplier 
estimates improve on the bounds and convergence results in the literature.
\end{abstract}

\section{Introduction}

\index{\QAPp, quadratic assignment problem}

We revisit and strengthen splitting methods for solving
doubly nonnegative,
\DNNp, relaxations of the quadratic assignment problem, \QAPp. 
We use a modified 
\textdef{restricted contractive Peaceman-Rachford splitting method, \rPRSMp}
approach. We obtain strengthened bounds from improved
lower and upper bounding techniques, and from strengthened
dual multiplier estimates. We compare with recent results 
in~\cite{OliveiraWolkXu:15}. In addition, we provide a new derivation of 
facial reduction, \FRp, and
the gangster constraints, and show the strong connections between them.

\index{doubly nonnegative, \DNNp}
\index{\DNNp, doubly nonnegative}

The \textdef{quadratic assignment problem, \QAPp}, is one of the fundamental 
combinatorial optimization problems in the fields of optimization and 
operations research, and includes many fundamental applications. 
It is arguably one of the hardest of the NP-hard problems.
The \QAP models real-life problems such as facility location.
Suppose that we are given a set of $n$ facilities and a set of $n$ locations. 
For each pair of locations $(s,t)$ a distance $B_{st}$ is specified, and for each 
pair of facilities $(i,j)$ a weight or flow $A_{i,j}$ is specified, e.g.,~the 
amount of supplies transported between the two facilities. 
In addition, there is a location 
(building) cost $C_{is}$ for assigning a facility $i$ to a specific 
location $s$. The problem is to assign each facility to a 
distinct location with the goal of minimizing the sum over all 
facility-location pairs of the distances between locations
multiplied by the corresponding flows between facilities, along with the
sum of the location costs. Other applications include:
scheduling, production, 
computer manufacture (VLSI design), chemistry (molecular conformation), 
communication, and other fields, see 
e.g.,~\cite{Els77,GeoGra76,ugi1979neue,krarup1978computer,heffley1977assigning}. 
Moreover, many classical combinatorial optimization problems, including 
the traveling salesman problem, maximum clique problem, and graph partitioning 
problem, can all be expressed as a \QAPp, 
see e.g.,~\cite{BhatiRasool:14,PardWolk:94}.  For more information about 
\QAPp, we refer the readers to~\cite{MR1490831,prw:93}.

That the \QAP \Cref{eq:qap} is NP-hard is given in~\cite{garjo}.
The cardinality of the feasible set of permutation
matrices $\Pi$ is $n!$ and it is 
known that problems typically have \emph{many} local minima.
Up to now, there are three main classes of methods for solving \QAPp.
The first type is heuristic algorithms, such as genetic algorithms,
e.g.,~\cite{Dre03}, ant systems~\cite{gambardella1999ant} and meta-heuristic
algorithms, e.g.,~\cite{bashiri2012effective}. 
These methods usually have short running times and often give optimal or near-optimal solutions. 
However the solutions from heuristic algorithms are not reliable and the
performance can vary depending on the type of problem.
The second type is branch-and-bound algorithms.
Although this approach gives exact solutions, it can be very time
consuming and in addition requires strong bounding techniques. 
For example, obtaining an exact solution using the
branch-and-bound method for $n=30$ is still considered to be
computationally challenging.
The third type is based on semidefinite programming, \SDPp. 
Semidefinite programming is proven to have successful implementations and 
provides tight relaxations, see~\cite{AnsBri:00,KaReWoZh:94}.
There are many well-developed \SDP solvers based on e.g.,~interior point 
methods, e.g.,~\cite{SaVaWo:97,AnjosLasserre:11,mitchell1998interior}. 
However, the running time of the interior point methods do not scale
well, and the \SDP relaxations become very large for the \QAPp.
In addition, adding additional polyhedral constraints 
such as interval constraints,
can result in having $O(2n^2)$ constraints, a prohibitive number for
interior point methods.

\index{\ADMMp, alternating direction method of multipliers}

Recently, Oliveira at el.,~\cite{OliveiraWolkXu:15} use an
\textdef{alternating direction method of multipliers, \ADMMp}, to solve
a  facially reduced, \FRp,  \SDP relaxation. The \FR 
allows for a natural splitting of variables between the \SDP cone and
polyhedral constraints.
The algorithm provides competitive lower and upper bounds for \QAPp. 
In this paper, we modify and improve on this work.

\subsection{Background}
It is known e.g.,~\cite{edw77copy}, that many of the \QAP 
models, such as the facility location problem,
can be formulated using the \textdef{trace formulation}:   
\begin{equation}\label{eq:qap}
p^*_{\QAPp}:=\min_{X \in \Pi}\langle AXB-2C, X\rangle,
\end{equation} 
where $A,B\in\mathbb{S}^n$ are real symmetric $n\times n$ matrices, $C$ is a real $n\times n$ matrix, $\langle \cdot, \cdot\rangle$ denotes the trace 
inner product, i.e.,~$\langle Y,X\rangle=\trace  (YX^T)$, and 
$\Pi$ denotes the set of $n\times n$ permutation matrices.
\index{$\Pi$, permutation matrices}
\index{permutation matrices, $\Pi$}
\index{trace inner product, $\langle Y,X\rangle=\trace  YX^T$}
\index{$\langle Y,X\rangle=\trace  YX^T$, trace inner product}

We use the following notation from~\cite{OliveiraWolkXu:15}.
We denote the \textdef{matrix lifting}
\index{vectorization of $X$ columnwise, $\kvec(X)$}
\index{$\kvec(X)$, vectorization of $X$ columnwise}
\begin{equation}\label{eq:lift}
Y:=\left (
\begin{array}{c}
1\\ x \end{array} 
\right )
(1~~x^T)\in \mathbb{S}^{n^2+1},\quad x = \vec(X) \in \R^{n^2},
\end{equation}
where $\kvec(X)$ is the vectorization of the matrix $X\in \Rnn$,
columnwise. Then $Y \in \textdef{$\Sc^{n^2+1}_+$}$,
the space of real symmetric positive semidefinite matrices of order
$n^2+1$, and the rank,
$\rank(Y)=1$. Indexing the rows and columns of $Y$ from $0$ to $n^2$,
we can express $Y$ in \Cref{eq:lift} using a block representation as 
follows:
\begin{equation}\label{eq:blocked}
Y=
\begin{bmatrix}
	Y_{00} & \bar{y}^T \\
	\bar{y}   &  \bY
\end{bmatrix}, \quad
\textdef{$\bar y$}=
\begin{bmatrix}
Y_{(10)}\\
Y_{(20)}\\
\vdots\\
Y_{(n0)}
\end{bmatrix},
\quad \hbox{and}\quad 
\textdef{$\bY = xx^T$}= \begin{bmatrix}
\bY_{(11)} & \bY_{(12)} & \cdots & \bY_{(1n)}\\
\bY_{(21)} & \bY_{(22)} & \cdots & \bY_{(2n)}\\
\vdots & \ddots & \ddots & \vdots\\
\bY_{(n1)} & \ddots& \ddots & \bY_{(nn)}
\end{bmatrix},
\end{equation}
where
\[
\textdef{$\bY_{(ij)}=X_{:i}X_{:j}^T$}\in \R^{n\times n},\, \forall i,j=1,\ldots,n, \
Y_{(j0)}\in\mathbb{R}^n,\forall j=1,\ldots,n,
\text{  and  } \ x \in \R^{n^2}.
\]
Let
\index{$B\otimes A$, Kronecker product}
\index{Kronecker product, $B\otimes A$}
\index{$L_Q=
\begin{bmatrix}
0 & -(\kvec(C)^T)\\
-\kvec(C) & B\otimes A
\end{bmatrix}$}
\[
L_Q=
\begin{bmatrix}
0 & -(\kvec(C)^T)\\
-\kvec(C) & B\otimes A
\end{bmatrix},
\]
where $\otimes$ denotes the Kronecker product.
With the above notation and matrix lifting, we can reformulate 
the \QAP \Cref{eq:qap} equivalently as
\begin{equation}
\label{eq:qapequal}
\begin{array}{rl}
p^*_{\QAPp}=
	\min & \langle AXB-2C, X\rangle = \langle  L_Q, Y \rangle \\
			\text{s.t.} 
			& Y:=
			\begin{pmatrix}
			1\\
			x
		\end{pmatrix}
			\begin{pmatrix}
	1\cr x
		\end{pmatrix}^T \in \Snt_+\\
	& X = \Mat(x)\in \Pi,
	\end{array}
\end{equation}
where \textdef{$\Mat = \kvec^{*}$}.

In~\cite{KaReWoZh:94},
Zhao et al. derive an \SDP relaxation as the dual of
the Lagrangian relaxation of a quadratically constrained version
of \cref{eq:qapequal}, i.e.,~the constraint that $X\in \Pi$ is replaced
by quadratic constraints, e.g.,
\[
\|Xe-e\|^2=\|X^Te-e\|^2=e, \, X\circ X = X, \, X^TX=XX^T=I,
\] 
where $\circ$ is the Hadamard product and $e$ is the vector of all ones.
After applying the so-call \emph{facial reduction} technique
to the \SDP relaxation, the variable $Y$ is
expressed as $Y={\hV}R{\hV}^T$, for some full column rank
matrix $\widehat{V}\in \R^{(n^2+1)\times ((n-1)^2+1)}$ defined below
in \Cref{sect:FRnDNN}. The \SDP then takes 
on the smaller, greatly simplified form:
\begin{equation}
\label{eq:sdp1}
	\begin{array}{cl}
	\min\limits_{R} & \langle \widehat{V}^T L_Q\widehat{V}, R\rangle \\
	\text{s.t. } & \GG_{\bar J}(\widehat{V}R\widehat{V}^T) = u_{0}\\
	& R\in \mathbb{S}_+^{(n-1)^2+1}.
	\end{array}
\end{equation}
The linear transformation $\GG_{\bar J}(\cdot )$ is called
the \emph{gangster operator} as it fixes certain matrix elements of the
matrix, and $u_{0}$ is the first unit vector and so all but the first
element are fixed to zero.
The Slater constraint qualification, strict feasibility, holds for 
both \Cref{eq:sdp1} and its dual, see~\cite[Lemma 5.1, Lemma 
5.2]{KaReWoZh:94}. We refer to 
\cite{KaReWoZh:94} for details on the derivation of this 
\textdef{facially reduced \SDPp}. 

\index{${\hat V}$}

We now provide the details for ${\hV}$, 
the gangster operator 
$\GG_{\bar J}$, and the \textdef{gangster index set, $\bar J$}. 
\index{$\bar J$, gangster index set} 
\begin{enumerate}
	\item 
Let $\hY$ be the barycenter of the set of feasible lifted $Y$
\cref{eq:blocked} of rank one for the \SDP
relaxation of \cref{eq:qapequal}. Let the matrix 
\textdef{$\widehat{V} $}$\in \R^{(n^2+1)\times((n-1)^2+1)}$ 
have orthonormal columns that span the range of $\hY$.\footnote{There
are several ways of constructing such a matrix $\hV$. One way is
presented in \Cref{prop:widehatVshort}, below.
}
Every feasible $Y$ of the \SDP relaxation is contained in the 
\textdef{minimal face, $\mathcal{F}$} of $\Sc^{n^2+1}_+$:
\index{$\mathcal{F}, $minimal face}
	\begin{equation*}
		\mathcal{F} = {\hV}
		\mathbb{S}_+^{(n-1)^2+1} {\hV}^T \unlhd \Sc^{n^2+1}_+.
	\end{equation*}
	
\index{gangster operator, $\GG_{\bar J}$} 
\index{$\GG_{\bar J}$, gangster operator} 
	\item The gangster operator 
is the linear map $\GG_{\bar J} : \mathbb{S}^{n^2+1} \to \R^{|\bar J|}$ 
defined by
	\begin{equation}
	\label{def:gangster_operator}
		\GG_{\bar J}(Y) = Y_{\bar J} \in \R^{|\bar J|},
	\end{equation}
	where $\bar J$ is a subset of (upper triangular) matrix indices of $Y$. 
\begin{remark}
By abuse of notation, we also consider the gangster operator 
as a linear map from $\mathbb{S}^{n^2+1}$ to $\mathbb{S}^{n^2+1}$,
depending on the context.
	\begin{equation}
	\label{eq:alt_gangster}
	\GG_{\bar J}:\mathbb{S}^{n^2+1}\rightarrow \mathbb{S}^{n^2+1}, \quad 
			\left[ \GG_{\bar J}(Y)\right]_{ij}
			=\left\{
			\begin{array}{cl}
				Y_{ij}		& \hbox{if}\,(i,j)\in \bar 
				J\;\hbox{or}\;(j,i)\in \bar J, \\
				0	& \hbox{otherwise},
			\end{array}\right.	
	\end{equation}
Both formulations 
	of $\GG_{\bar J}$ are used for defining a constraint which ``shoots 
	holes" 
	in the 
matrix $Y$ with entries indexed using $\bar J$. Although the latter 
formulation
is more explicit, it is not surjective and is not used in the implementations.
\end{remark} 
	
\item  
The gangster index set $\bar J$ is defined 
to be the union of the top left index $(00)$ with the set of indices
$J$ with $i<j$ in the submatrix $\bY \in \mathbb{S}^{n^2}$
corresponding to:
\begin{equation} \hspace{-3cm}
\label{def:Gangster_indices}
	\begin{array}{l}
		\text{(a) the off-diagonal elements in the $n$ diagonal blocks in $\bY$
		in \Cref{eq:blocked} } ;  \\
		\text{(b) the diagonal elements in the off-diagonal blocks in $\bY$
		in  \Cref{eq:blocked} }.
	\end{array}
\end{equation}
Many of the constraints that arise from the index set $J$ are
redundant. We could remove the 
indices in the submatrix $\bY \in \mathbb{S}^{n^2}$ corresponding to all 
the diagonal positions of the last column of blocks and
the additional $(k-2,k-1)$ block. 
In our implementations we take advantage of redundant constraints when
used as constraints in the subproblems.

\item 
The notation $u_{0}$ in \Cref{eq:sdp1} denotes a vector in $\{0,1\}^{|\bar 
J|}$ 
with $1$ only in the first coordinate, i.e.,~the $0$-th unit vector.
Therefore \Cref{eq:sdp1} forces all 
the values of $\widehat{V}R\widehat{V}^T$ corresponding to the indices in 
$\bar J$ to be zero. It also implies that the first entry of $\GG_{\bar J} 
(\widehat{V}R\widehat{V}^T) $ is equal to 1, which reflects the fact that 
$Y_{00} =1$ from $\Cref{eq:blocked}$. Using the alternative 
definition of $\GG_{\bar J}$ in \Cref{eq:alt_gangster}, the equivalent 
constraint is $\GG_{\bar J}(Y) = E_{00}$ where \textdef{$E_{00}$} $\in 
\mathbb{S}^{n^2+1}$ is the $(0,1)$-matrix with $1$ only in the  
$(00)$-position. 

\index{$u_0$, first unit vector}
\index{first unit vector, $u_0$}

%The matrix \textdef{$E_{00}$} $\in \mathbb{S}^{n^2+1}$ is the $(0,1)$-matrix with only $1$ on the  $(00)$-position.   
\end{enumerate}        
  										
Since interior point solvers do not scale well, especially when
nonnegative cuts are added to the \SDP relaxation in \Cref{eq:sdp1}, 
Oliveira et al.~\cite{OliveiraWolkXu:15} propose using an \ADMM
approach. They introduce nonnegative cuts (constraints) and
obtain a \textdef{doubly nonnegative, \DNNp}, model.
The \ADMM approach is further motivated by the natural splitting of variables 
that arises with facial reduction:
\index{\DNNp, doubly nonnegative}
\begin{equation}
\label{FRrelaxation}
(\DNNp) \qquad \begin{array}{cl}
\min\limits_{R,Y} & \langle L_Q,Y \rangle \\
\text{s.t.} & \GG_{\bar J}(Y) = u_0\\
& Y = \widehat{V} R \widehat{V}^T\\
& R \succeq 0\\
& 0\leq Y \leq 1.
\end{array}
\end{equation}
The output of \ADMM is used to compute lower and upper bounds to 
the original \QAP \Cref{eq:qap}. For most instances in 
QAPLIB\footnote{\url{http://coral.ise.lehigh.edu/data-sets/qaplib/qaplib-problem-instances-and-solutions/}},
\cite{OliveiraWolkXu:15} obtain competitive lower and upper bounds for the 
\QAP using \ADMMp. 
And in several instances, the relaxation and 
bounds provably find an optimal permutation matrix.

\subsubsection{Further Notation}
We let $\Rn$ denote the usual Euclidean space of dimension $n$.
We use \textdef{$\mathbb{S}^n$} to denote the space of real symmetric matrices
of order $n$.
We use \textdef{$\Snp$} (\textdef{$\Snpp$}, resp.) to denote the cone of 
$n$-by-$n$ positive semidefinite (definite) matrices.
We write \textdef{$X \succeq 0$} if $X\in \Snp$ and \textdef{$X \succ
0$} if $X\in \Snpp$.
Given $X\in \Rnn$, we use $\trace (X)$ to denote the trace of $X$. 
We use \textdef{$\circ$} to denote the Hadamard (elementwise) product.
Given a matrix $A\in R^{m\times n}$, we use $\range(A)$ and $\Null(A)$ to 
denote the range of $A$ and the null space of $A$, respectively.
\index{$\range(A)$, range space}
\index{range space$, \range(A)$}
\index{$\Null(A)$, nullspace}
\index{nullspace$, \Null(A)$}

We denote $u_{0}$ to be the unit vector of appropriate dimension with
$1$ in the first coordinate. By abuse of notation, for $n \geq 1$, 
$e_{n}$ denotes the vector of all ones of dimension $n$.
 $E_{n}$ denotes the $n \times n$ matrix of all ones. We omit the 
 subscripts of $e_n$ and $E_n$ when the dimension is clear. 
\index{$E$, matrix of ones}
\index{matrix of ones, $E$}
\index{$e$, vector of ones}
\index{vector of ones, $e$}

\subsection{Contributions and Outline}
We begin in \Cref{sect:model} with the modelling and theory.
We first give a new joint derivation of the 
so-called gangster constraints and the facial reduction procedure. 
We then propose a strengthened model of \cref{FRrelaxation}, by 
imposing a trace constraint to the variable $R$, and use this for 
deriving a modified 
\textdef{restricted contractive Peaceman-Rachford splitting method, \rPRSMp}
for solving the strengthened model. 
We improve lower bounds presented in~\cite{OliveiraWolkXu:15} by utilizing the trace constraint added to the variable $R$. 
We also adopt a randomized perturbation approach to improve upper
bounds. In addition, we improve the running time with new dual variable updates as well as adopting additional termination conditions. 
Our numerical results in \Cref{sect:numeric} show significant 
improvements over the previous results in~\cite{OliveiraWolkXu:15}. 

\section{The \DNN Relaxation}
\label{sect:model}
In this section we present details of our
\textdef{doubly nonnegative, \DNNp},
relaxation of the \QAPp. This is related to the \SDP relaxation derived in
\cite{KaReWoZh:94} and the \DNN relaxation in~\cite{OliveiraWolkXu:15}.
Our approach is novel in that we see the gangster constraints and facial 
reduction arise naturally from the 
relaxation of the row and column sum constraints for $X \in \Pi$.

\subsection{Novel Derivation of \DNN Relaxation}
The \SDP relaxation in~\cite{KaReWoZh:94} starts with the Lagrangian
relaxation  (dual) and forms the dual of this dual. Then redundant
constraints are deleted. We now look at a direct approach for finding
this \SDP relaxation.
\subsubsection{Gangster Constraints}

\index{$\ZZ$, zero-one matrices}
\index{zero-one matrices, $\ZZ$}
\index{$\DD$, doubly stochastic}
\index{doubly stochastic, $\DD$}
\index{$\DD_e$, row and column sums equal one}
\index{row and column sums equal one, $\DD_e$}

Let $\DD_e$ and $\ZZ$ be the sets of row and column sums equal one
matrices, and the set of binary matrices, respectively:
\[
\begin{array}{ccl}
\DD_e &:=& \{ X \in \R^{n \times n} : Xe=e, X^{T}e = e   \},  \\
\ZZ &:=& \{ X \in \R^{n \times n} : X_{ij} \in \{0,1\} , \  \forall i,j \in 
\{1,...n\} \}.
\end{array}
\]
We let $\DD = \DD_e\cap \{X\geq 0\}$ denote the \emph{doubly stochastic matrices}. The classical
Birkhoff-von Neumann Theorem~\cite{MR0054920,Birk:46} states
that the permutation matrices are the extreme points of $\DD$.
This leads to the well-known conclusion that
the set of $n$-by-$n$ permutation matrices, $\Pi$, is equal to 
the intersection:
\begin{equation}
\label{eq:reprowcolsumzo}
\Pi =  \DD_e \cap \ZZ.
\end{equation}
It is of interest that the representation in \Cref{eq:reprowcolsumzo}
leads to \underline{\emph{both}} the gangster constraints and facial
reduction for the \SDP relaxation on the lifted variable $Y$ in
\Cref{eq:blocked}, and in particular on $\bY$.
Not only that, but the row-sum constraints $Xe = e$, along with
the $0$-$1$ constraint, expressed as  $X \circ X = X$, give rise to the
constraint that the diagonal elements of the off-diagonal blocks of $\bY$
are all zero; while the column-sum constraint $X^{T}e = e$ along with the 
$0$-$1$ constraints give rise to the constraint that the off-diagonal 
elements of the diagonal blocks of $\bY$ are all zero.
The following well-known \Cref{lem:trace_relation} 
about complementary slackness is useful.
\begin{lem} \label{lem:trace_relation}
	Let $A,B \in \mathbb{S}^n$. If $A$ and $B$ have nonnegative entries, 
	then $\langle A,B \rangle = 0  \iff A\circ B = 0$.
\end{lem}
\begin{proof}
This is clear from the definitions of $A,B$.
\end{proof}

The following \Cref{thm:new_gangster_derivation} and 
\Cref{thm:constraint_derivation} together show how the representation
of $\Pi$ in \Cref{eq:reprowcolsumzo} gives rise to the gangster
constraint on the lifted matrix $Y$ in \Cref{eq:lift}. 
We first find (Hadamard product)
\textdef{exposing vector}s in \Cref{thm:new_gangster_derivation} for
lifted zero-one vectors.
\begin{lemma}[exposing vectors] \label{thm:new_gangster_derivation}
Let $X \in \ZZ$ and let $x:= \kvec(X)$. Then the following hold:
\begin{enumerate}
\item \label{new_derivation_item1}
$Xe_n = e_n \implies  [(e_ne_n^T \otimes I_n) - I_{n^2} ]  \circ xx^{T} = 
0$;
\item \label{new_derivation_item2}
$X^Te_n = e_n \implies  [(I_n\otimes e_ne_n^T) - 
I_{n^2} ] 
\circ xx^{T} = 0$.
\end{enumerate}
\end{lemma}
\begin{proof}
We first show \Cref{new_derivation_item1}. 
Let $X \in \ZZ$ and $Xe_n = e_n$. 
We note that $X \in \ZZ \iff x \circ x -x = 0$ and 
\[
Xe_n = e_n \iff I_nXe_n = e_n \iff (e_n^{T} \otimes I_n) x = e_{n}.
\]
We begin by multiplying both sides by $(e_n^{T} \otimes I)^T= e_n \otimes I$:
\[
\begin{array}{rrcl}
& (e_n^{T} \otimes I_n) x & =& e_{n} \\
\implies& (e_n \otimes I_n)(e_n^{T} \otimes I_n) x & = &(e_n \otimes I_n 
)e_n = e_{n^2} \\
\implies &[(e_n \otimes I_n)(e_n^{T} \otimes I_n) - I_{n^2} ] x & =& 
e_{n^2}  - x \\
\implies& [(e_ne_n^T \otimes I_n) - I_{n^2} ] xx^{T} 
& =& 
e_{n^2}x^{T}  - xx^{T}  \\
\implies &\trace  \left( [(e_ne_n^T \otimes I_n) - 
I_{n^2} ] \ 
xx^{T} \right)  & =& \trace  (e_{n^2}x^{T}  - xx^{T} ).
\end{array}
\]
Since $x \circ x = x$, we have $\trace  (e_{n^2}x^{T}  - 
xx^{T} ) = 0$. Therefore, it holds that 
$$\trace  \left(  [(e_ne_n^T \otimes I_n) - I_{n^2} ]  \ 
xx^{T} \right)=0.$$
We note that $ [(e_ne_n^T \otimes I_n) - I_{n^2} ] $ and $xx^{T}$ are both 
symmetric and nonnegative.
Hence, by \Cref{lem:trace_relation},  
we get
\[
[(e_ne_n^T \otimes I_n) - I_{n^2} ] \circ xx^{T} = 0.
\]
The proof for \Cref{new_derivation_item2} follows by using a similar argument.
% commented the proof for item 2.
%The proof for \Cref{new_derivation_item2} is similar. We now let $X  \in \ZZ$ be such that $X^Te_n = e_n$. Then we have\[X^Te_n = e_n \iff e_n^TXI_n = e_n \iff ( I_n\otimes e_n^{T}) x = e_{n}.\]By multiplying $(I\otimes e_n)$ to the equality above, we get \[[(I_n\otimes e_n)(I_n \otimes e_n^T) - I_{n^2} ] \circ xx^{T} = 0. \]
\end{proof}

\begin{cor} \label{thm:constraint_derivation}
Let $X\in\Pi$, and let $Y$ satisfy \Cref{eq:lift}.
Let $\GG_{\bar J}, \bar{J}$ be defined in 
\cref{def:gangster_operator} and \Cref{def:Gangster_indices}. Then the 
following hold:
\begin{enumerate}
\item \label{item:gangster_on_Y}
$\GG_{\bar J}(Y) = u_0$;
\item \label{item:Yin01}
$0\leq Y\leq 1, \, Y\succeq 0, \, \rank(Y)=1$.
\end{enumerate}
\end{cor}

\begin{proof}
Note that 
\begin{itemize}
	\item the matrix $(e_ne_n^T \otimes I_n) - I_{n^2} $  
	has 
	nonzero 
	entries on 
	the diagonal elements of the off-diagonal blocks;
	\item the matrix $(I_n\otimes e_ne_n^T) - 
	I_{n^2}$ has nonzero entries on the off-diagonal elements of 
	the diagonal 
	blocks.
\end{itemize}
	Therefore, \Cref{thm:new_gangster_derivation}, the 
	definition of the gangster indices $\bar J$ in 
	\Cref{def:Gangster_indices}, and the structure of $Y$ in 
	\Cref{eq:lift}, jointly give $\GG_{\bar J}(Y)=u_0$, i.e., 
	\Cref{item:gangster_on_Y} holds.  
	\Cref{item:Yin01} follows from
\Cref{eq:reprowcolsumzo} and the structure of $Y$ in \Cref{eq:lift}.
\end{proof}

So far, we have shown that the 
representation $\Pi = \DD_e \cap \ZZ$ gives rise to the gangster 
constraint and the polyhedral constraint on the variable $Y$ given in 
\cref{FRrelaxation}. Therefore, replacing the constraints in  
\Cref{eq:qapequal} by the items in \Cref{thm:constraint_derivation}, 
and discarding the \emph{hard} rank-one constraint, we 
get the following \SDP relaxation:
\begin{equation}
\label{eq:sdp2}
\begin{array}{rl}
p^*_{\QAPp}\geq\min\limits_{Y} &  \langle L_Q, Y\rangle \\
\text{s.t. } & \GG_{\bar J}(Y) = u_{0}\\
& 0\leq Y\leq 1\\
& Y\succeq 0.
\end{array}
\end{equation}

\subsubsection{Facially Reduced \DNN Relaxation}
\label{sect:FRnDNN}

Next, we explore the derivation for the facial 
reduction constraint $Y = \widehat{V} R \widehat{V}^T$ in 
\cref{FRrelaxation}. 
As for the derivation of the gangster constraint, it arises from
consideration of an exposing vector. We define 
\begin{equation}
\label{def:H}
\textdef{$H$} := \begin{bmatrix} 
e_n^{T} \otimes I_n \cr I_n\otimes e_n^{T}
\end{bmatrix} \in  \R^{2n\times n^2},
\end{equation}
and 
\begin{equation}\label{eq:Kexp}
K :=\begin{bmatrix} 
-e_{n^2}^T\\ \vspace{-0.7em} \\
H^T
\end{bmatrix}
\begin{bmatrix}
-e_{n^2} & H
\end{bmatrix}
=
\begin{bmatrix} n^2 & -2e_{n^2}^T \\ -2e_{n^2} & H^TH 
\end{bmatrix} \in \mathbb{S}^{n^2+1}.
\end{equation}
We note that $H$ arises from the linear equality constraints $Xe=e, X^Te =e$. The matrix 
$H$ in \Cref{def:H} is the well-known matrix in the linear
assignment problem with $\rank(H) = 2n-1$ and the rows sum up to $2e_{n^2}^T$. 
Then $\rank(K)=2n-1$ as well. Moreover, the following 
\Cref{lem:equivlinconstr} is 
clear.
\begin{lemma} 
\label{lem:equivlinconstr}
Let $H$ be given in \cref{def:H}; and let 
\[
X\in \Rnn,\, x=\kvec(X),\, Y_x = 
\begin{pmatrix}
1 \\ x
\end{pmatrix}
\begin{pmatrix}
1 \\ x
\end{pmatrix}^T.
\]
Then
\[
\begin{array}{rcl}
Xe=e, X^Te=e
&\iff   &
Hx = e 
\\&\iff   &
\begin{pmatrix}
1 \\ x
\end{pmatrix}^T
\begin{pmatrix}
-e^T \\ H^T
\end{pmatrix} = 0
\\&\implies   &
\begin{pmatrix}
1 \\ x
\end{pmatrix}
\begin{pmatrix}
1 \\ x
\end{pmatrix}^T
\begin{pmatrix}
-e^T \\ H^T
\end{pmatrix}
\begin{pmatrix}
-e^T \\ H^T
\end{pmatrix}^T = 0
\\&\iff   &
Y_x K=0.  
\end{array} 
\]
\end{lemma}
\vspace{-1.8em} 
\hfill$\qedsymbol$

From \Cref{lem:equivlinconstr}, $K$ is an \textdef{exposing 
vector} for all
feasible $Y_x$ and so for all feasible $Y$ in \Cref{eq:sdp2}, see 
e.g.,~\cite{DrusWolk:16}.
Then we can choose a full column rank $\hV$ with the range equal to the
nullspace of $K$ and obtain \textdef{facial reduction}, i.e.,~all feasible $Y$ for
the \SDP relaxation satisfy
\[
Y \in \hV \mathbb{S}_{+}^{(n-1)^2+1} \hV^T \unlhd \mathbb{S}_+^{n^2+1}.
\]
%\index{$\kMat(x)$, matrix reshape}
%\index{matrix reshape, $\kMat(x)$}

There are clearly many choices for $\hV$. We present one in 
\Cref{prop:widehatVshort} that is studied in~\cite{KaReWoZh:94}. In our 
work we use $\hV$ that have orthonormal columns as 
in~\cite{OliveiraWolkXu:15}, i.e.,~ $\widehat V^T\widehat V=I$.

\begin{prop}[\!\!\cite{KaReWoZh:94}]
\label{prop:widehatVshort}
Let 
\[{\hV}=\left[
\begin{array}{c|c}
1 & 0 \\ \hline \frac{1}{n} e_{n^2} & V_e\otimes V_e
\end{array}
\right] \in \R^{(n^2+1) \times ((n-1)^2+1)},
\quad
V_e = 
\begin{bmatrix}
I_{n-1} \\ - e_{n-1}^T
\end{bmatrix} \in \R^{n\times (n-1)},
\]
and let $K$ be given as in \Cref{eq:Kexp}.
Then we have $\mathcal{R}({\hV}) =\mathcal{R}(K)$.
\qed
\end{prop}
Our \DNN relaxation has the lifted $Y$
from \cref{eq:lift,eq:qapequal} and the \FR variable $R$
from \cref{eq:sdp1}. The relation between $R,Y$ provides the natural
\textdef{splitting}:
\begin{equation}\label{eq:facred3}
\begin{array}{rccl}
p_{\DNN}^*&=&
\min & \langle L_Q,Y \rangle \\
&&\text{s.t.}& 
\GG_{\bar J}(Y) = u_0 \\
&&& 
Y={\hV}R{\hV}^T \\
&&& R \succeq 0\\
&&& 0\leq Y\leq 1.
\end{array}
\end{equation}

A strictly feasible $\hat R\succ 0$  for 
the facially 
reduced \SDP relaxation is
given in~\cite{KaReWoZh:94}, based on the barycenter $\hat Y$ of the lifted
matrices $Y$ in \cref{eq:lift}. Therefore, $0<\hat Y_{\bar J^c}<1$
and this pair $(\hat R,\hat Y)$ is strictly feasible in \cref{eq:facred3}.
%see~\Cref{prop:Robinson}.

\subsection{Adding Redundant Constraints}
We continue in this section with some redundant constraints 
for the model \cref{eq:facred3} that are useful in the subproblems.
\subsubsection{Preliminary for the Redundancies}
Before we present the redundant constraints for \Cref{eq:facred3}. 
We first recall two linear transformations defined in 
\cite{KaReWoZh:94}.     % and a useful lemma \Cref{lem:arrow}. 
%%%\begin{lemma}[\!\!{\cite[Lemma 3.1]{KaReWoZh:94}}]
%%%	\label{lem:arrow}
%%%  Let $R\in\mathbb{S}^{(n-1)^2+1}$ be arbitrary and let 
%%%  \[
%%%  Y={\hV}R{\hV}^T,
%%%  \]
%%%  where ${\hV}$ is given in \Cref{prop:widehatVshort}. Then, using
%%%  the block notation of \Cref{eq:blocked}, we have
%%%  \begin{enumerate}
%%%  	\item\label{item:sum1row} $Y_{00}=R_{00}$, 
%%%  	$Y_{(0j)}e_n=R_{00}$ 
%%%  	and 
%%%  	$\sum_{j=1}^nY_{(0j)}
%%%  	=R_{00}e_n^T$.
%%%  	\item\label{item:arrowdiag} $Y_{(0j)}=e_n^TY_{(ij)}$, for 
%%%  	$i,j=1,\ldots,n$.
%%%  	\item\label{item:diagblock} $\sum_{i=1}^n Y_{(ij)}=e_nY_{(0j)}$ 
%%%  	and 
%%%  	$\sum_{i=1}^n\diag(Y_{(ij)})=Y_{(j0)},$ for $j=1,\ldots, n$.
%%%  \end{enumerate}
%%%\end{lemma}

\begin{definition}[\!\!\!{\cite[Page 80]{KaReWoZh:94}}]
\label{def:bdiag-odiag}
	Let $Y\in\mathbb{S}^{n^2+1}$ be blocked as in \cref{eq:blocked}. We  
	define
	the linear transformation $\bodiag (Y):\mathbb{S}^{n^2+1}\rightarrow 
	\mathbb{S}^n$ by the sum of the $n$-by-$n$ diagonal blocks 
	of $Y$, i.e., 
	\begin{equation*}\label{eq:bdiag}
	\bodiag (Y):=\sum_{k=1}^nY_{(k\,k)}\in\mathbb{S}^n.
	\end{equation*}
We	define the linear transformation 
 	$\oodiag (Y):\mathbb{S}^{n^2+1}\rightarrow 
	\mathbb{S}^n$ by the trace of
	the block $\bY_{(ij)}$, i.e.,
	 \begin{equation*}\label{eq:odiag}
	 \oodiag (Y):=
       \left(\trace \left( \bY_{(ij)}\right)\right)_{ij} \in\mathbb{S}^n.
	 \end{equation*}
\end{definition}
With \Cref{def:bdiag-odiag}, the following lemma can be derived from 
\cite[Lemma 3.1]{KaReWoZh:94}.
\begin{lemma}[\!\!\!{\cite[Lemma 3.1]{KaReWoZh:94}}]
\label{prop:redundant_from_Zhao}
Let $V$ be any full column rank matrix such that $\range(V) = \range(\whV)$, 
where $\whV$ is given in \Cref{prop:widehatVshort}.
Suppose $Y = V R V^T$ and  $\cG_{\bar J} (Y) = u_0$ hold.
Then the following hold:
\begin{enumerate}
\item \label{item:arrow_redundant}
The first column $Y$ is identical with the diagonal of $Y$.
\item \label{item:bdiag_odiag_redundant}
$\bodiag(Y) = I_n$ and $ \oodiag(Y)= I_n$.  \qed
\end{enumerate} 
\end{lemma}

\subsubsection{Adding Trace $R$ Constraint}

The following \Cref{prop:trRredund} now shows that the constraint 
$\trace (R)=n+1$ in \Cref{def:R} is indeed
redundant. But, as mentioned, 
it is not redundant when the subproblems of \rPRSM are considered
as independent optimization problems. We take advantage of this in the 
corresponding $R$-subproblem and the computation of the lower bound of 
\QAPp. 

\begin{prop}
\label{prop:trRredund}
	The constraint $\trace (R) = n+1$ is redundant in \cref{eq:main}, i.e.,
	$Y={\hV} R{\hV}^T$, $R\succeq0$
	and $Y\in\mathcal{Y}$ yields that $\trace (R) = n+1$.
\end{prop}
\begin{proof}
By \Cref{prop:redundant_from_Zhao}, $\bodiag(Y) =I_n$ hold.
Then with $Y_{00}=1$, we see that $\trace (Y) = n+1$.
By cyclicity of the trace operator and $\whV^T \whV = I$, we see that 
\[
\trace (R)=\trace (R)\hV^T\hV=\trace \left(\hV R\hV^T\right)=\trace (Y)=n+1. 
\qedhere
\] %\end{enumerate}	
\end{proof}

\begin{remark}
	Note that we could add more redundant constraints to (\DNNp). For 
	example:
	\begin{equation*}
	\label{eq:mainequiv}
	\begin{array}{rcl}
	\textdef{$p^*_{\DNN}$}:=& \min\limits_{R,Y} &  \langle L_Q,Y \rangle\\
	&\text{s.t.} & Y = \widehat{V} R \widehat{V}^T\\
	&& \GG_{\bar J}(Y) = u_0 \\
	&& \GG_{\bar J}({\hV} R {\hV}) = u_0 \\
	&& R\succeq0\\
	&& 0\leq Y\leq 1.
	\end{array}
	\end{equation*}
	We could also add redundant constraints to the sets $\RR,\YY$ that are
	not necessarily redundant in the subproblems below, thus strengthening 
	the splitting approach.
	For example, we could use the so-called $\arrow,\bodiag,\oodiag$ 
	constraints that are defined and shown redundant 
	in~\cite{KaReWoZh:94}. 
	Moreover, 
from \Cref{item:bdiag_odiag_redundant} of 
	\Cref{prop:redundant_from_Zhao}, 
	$\Mat\left( \diag(\bY) \right)$ is doubly stochastic for a feasible 
	$Y$ to the model \eqref{eq:facred3}, where $\Mat$ is the adjoint of 
	the $\vec$ operator.
	Hence one may include an additional redundant constraint to the model 
	\eqref{eq:facred3}.
	Moreover, we could strengthen the relaxation by
	restricting each row/column (ignoring the first row/column) to be a
	multiple of a vectorized doubly stochastic matrix.
\end{remark} 

\subsection{Optimality Conditions for Main Model}
\label{sect:augmlagroptcond}
We now derive the main splitting model. We
define the cone and polyhedral constraints, respectively, as
\index{polyhedral constraints, $\YY$}
\index{$\YY$, polyhedral constraints}
\index{cone constraints, $\RR$}
\index{$\RR$, cone constraints}
\begin{equation}
\label{def:R}
\RR:= \left\{R\in\mathbb{S}^{(n-1)^2+1}: R\succeq0,\ 
        \trace  (R) = n+1 \right\},
\end{equation}
and 
\begin{equation}
\label{def:Y}
	\YY:= \left\{Y\in\mathbb{S}^{n^2+1}:\GG_{\bar J}(Y) =u_0,\ 0\leq
	Y\leq 1 \right\}.
\end{equation}
Replacing the constraints in \Cref{eq:facred3} with \Cref{def:R}
and \Cref{def:Y}, we obtain the following \DNN relaxation 
that we solve using \rPRSMp:
\index{\DNN relaxation of \QAP}
\begin{equation}
\label{eq:main}
(\DNN) \qquad
\begin{array}{rcl}
\textdef{$p^*_{\DNN}$}:=& \min\limits_{R,Y} &  \langle L_Q,Y \rangle\\
&\text{s.t.} & Y = \widehat{V} R \widehat{V}^T\\
&& R\in\RR\\
&& Y\in \YY.
\end{array}
\end{equation}
%We note that the model \cref{eq:main} satisfies the Robinson 
%constraint qualification. Note that the constraint
%mapping $Y-\hV R\hV$ is surjective. 
The Lagrangian function of model 
\Cref{eq:main} is:
\begin{equation*}
%\label{eq:lagrangian}
\LL (R, Y,Z) = \langle L_{Q}, Y \rangle  + \langle Z,Y - {\hV} R 
{\hV}^T \rangle.
\end{equation*}
\index{Lagrangian, $\LL$}
\index{$\LL$, Lagrangian}
The first order optimality conditions for the model \cref{eq:main} are:
\begin{subequations}\label{eq:optcond}
	\begin{alignat}{5}
	0&\in  -\widehat{V}^TZ\widehat{V}+\NN_{\RR}(R),
            &\quad \text{(dual $R$ feasibility)} \label{eq:DNNa} \\
	0&\in  L_Q+Z+\NN_{\YY}(Y),
            &\quad \text{(dual $Y$ feasibility)} 
\label{eq:DNNb}\\
	Y & =   \widehat{V}R\widehat{V}^T,
              \quad R \in \RR, \,  Y \in \YY,
            & \quad \text{(primal feasibility)} \label{eq:DNNc}
	\end{alignat}
\end{subequations} 
where the set $\NN_{\RR}(R)$ (resp. $\NN_{\YY}(Y)$) is the normal cone to
the set $\RR$ (resp. $\YY$) at $R$ (resp. $Y$).  
%Recall that, given any convex set $C$, for $x\in C$, the \emph{normal 
%cone 
%of $C$ 
%at $\bar x$, $\NN_{C}(\bar x)$}, is:
\index{$\NN_{C}(\bar x)$, normal cone of $C$ at $\bar x$}
\index{normal cone of $C$ at $\bar x$, $\NN_{C}(\bar x)$} 
%\[
%\NN_{C}(\bar x) = \{s: \langle s, x-\bar x\rangle \leq 0, \, \forall x 
%\in 
%C\}. 
%\]
%We also recall the following characterization:
%\[
%s\in \mathcal{N}_C(\bar{x})
%\iff \< (s+\bar{x})-\bar{x}, x-\bar{x} \rangle \le 0 , \ \forall x \in C
%\iff \bar{x} = \mathcal{P}_C ( \bar{x} +s), 
%\]
%where $\mathcal{P}_C(\cdot)$ is the projection onto the convex set $C$.
%Then the optimality conditions in \cref{eq:optcond} can be written as 
%follows.
%\index{$\mathcal{P}_C(\cdot)$, projection onto convex set $C$}
%\index{projection onto convex set $C$, $\mathcal{P}_C(\cdot)$}
By the definition of the normal cone, we can easily obtain the following 
\Cref{thm:charactoptMAIN}.
\begin{prop}[characterization of optimality for \cref{eq:main}]
\label{thm:charactoptMAIN}
The primal-dual $R,Y,Z$ are optimal for  \cref{eq:main} if, and only if,
\cref{eq:optcond} holds if, and only if,
\begin{subequations}\label{eq:optcondnew}
	\begin{alignat}{4}
	R &=  \mathcal{P}_\RR (R+\widehat{V}^TZ\widehat{V}) \label{eq:optcondnew_inR}  	\\
	Y& = \mathcal{P}_\YY (Y- L_Q - Z) \label{eq:optcondnew_inY}\\
	Y & =   \widehat{V}R\widehat{V}^T. \label{eq:optcondnew_pR}
	\end{alignat} 
\end{subequations} 
\end{prop}
We use \Cref{eq:optcondnew} as one of the stopping criteria of the \rPRSM 
in our numerical experiments. 

As in all optimization, the dual multiplier, 
here $Z$, is 
essential in finding
an optimal solution.
We now present properties on $Z$ that are exploited in our
algorithm in \Cref{sect:PRSM}.
\Cref{thm:LQzero} shows that there exists a dual multiplier $Z\in 
\mathbb{S}^{n^2+1}$ of the model \Cref{eq:main} that, except for the $(0,0)$-th entry, 
has a known diagonal, first column and first row. This allows for 
faster convergence in the algorithm in \Cref{sect:PRSM}.
\begin{theorem}
\label{thm:LQzero}
Let $(R^*,Y^*)$ be an optimal pair for~\Cref{eq:main}, and let
\[
\textdef{$\ZZ_A$} :=\left\{Z\in\mathbb{S}^{n^2+1}: Z_{i,i} = -(L_Q)_{i,i}, 
\ Z_{0,i}=Z_{i,0}= -(L_Q)_{0,i} , \ i=1,\ldots ,n^2 \right\}.
\]
Then there exists $Z^*\in\ZZ_A$ such that $(R^*,Y^*,Z^*)$ 
solves \cref{eq:optcond}.	
\end{theorem}
\begin{proof}
We define \index{$\YY_A$}$\YY_A:= \left\{Y\in\mathbb{S}^{n^2+1}:\GG_J(Y) = 
E_{00},\ 0\leq E_A \circ Y\leq 1 \right\}$, where
$E_A = \left[ \begin{array}{c|c}
1 & 0 \\ \hline 0 & E_{n^2} - I_{n^2}
\end{array} \right]$.
Namely, $\YY_A$ consists of the elements of $\YY$ after removing the polyhedral constraints on the diagonal and the first row and column.
Consider the following problem:
\begin{equation}\label{eq:main-D}
\min_{R,Y} \{ \langle L_Q, Y\rangle : Y= \widehat{V}R\widehat{V}^{T}, \ R \in \RR, \ Y \in \YY_A \}.
\end{equation}
Clearly, every feasible solution of \Cref{eq:main} is feasible for 
\Cref{eq:main-D}.
Consider a feasible pair $(R,Y)$ to \Cref{eq:main-D}.
By \Cref{item:bdiag_odiag_redundant}
of \Cref{prop:redundant_from_Zhao} and the 
positive semidefiniteness of $Y=  \widehat{V}R\widehat{V}^{T}$,
the elements of the diagonal of $Y$ are in the interval $[0,1]$. 
In addition, by \Cref{item:arrow_redundant} of \Cref{prop:redundant_from_Zhao},
the elements of the first row 
and 
column of $Y$  are also in the interval $[0,1]$.
Thus we conclude that $Y \in \YY$ and 
\Cref{eq:main} and \Cref{eq:main-D} are equivalent.

Let $(R^*,Y^*)$ be a pair of optimal solution to \Cref{eq:main-D}.
%We note that the pair $(R^\circ,Y^\circ)$ in 
%\Cref{eq:Rcirc} and \Cref{eq:Ycirc} of \Cref{lem:Jr}
%gives $Y^\circ 
%\in \YY\subseteq \YY_A$.
%Hence, \Cref{eq:main} and \Cref{eq:main-D} share the same Slater point 
%$R^\circ$.
Hence, there exists a $Z^*$  that satisfies the following characterization
of optimality:
\begin{subequations}
\label{eq:DNND}
\begin{alignat}{4}
	0 & \in -\widehat{V}^T Z^*{\hV} +\NN_{\RR}(R^*), 
\label{eq:DNN_Da} \\
	0 & \in L_Q + Z^*+\NN_{\YY_A}(Y^*), 
\label{eq:DNN_Db} \\
	Y^*& = \widehat{V}R^*\widehat{V}^T, \quad R^*\in \RR,\,
Y^* \in \YY_A. \label{eq:DNN_Dc}
\end{alignat}
\end{subequations} 
By the definition of the normal cone, we have
\[
0  \in L_Q + Z^*+\NN_{\YY_A}(Y^*)
\iff
\langle Y-Y^*,L_Q + Z^* \rangle \ge 0 , \ \forall\, Y\in \YY_A.
\]
Since the diagonal and the first column and row of $Y\in \YY_A$ except for the first element are unconstrained, we see that 
\[
(E_{n^2+1}-E_A) \circ (Z^*+L_Q) = 0,
\]
which implies that 
\[
Z_{ii}=-(L_Q)_{i,i},Z_{0,i}=Z_{i,0}=-(L_Q)_{0,i},\; i=1,\ldots,n^2, \ \text{ i.e., } Z^*\in\ZZ_A .
\]

In order to complete the proof, it suffices to show that the triple 
$(R^*,Y^*, Z^*)$	also solves~\cref{eq:optcond}.
We note that \Cref{eq:DNN_Da} and \Cref{eq:DNN_Dc} imply that 
\cref{eq:DNNa} and \cref{eq:DNNc} hold with $(R^*,Y^*, Z^*)$ in the place 
of $(R,Y,Z)$. 
In addition, since $Y^*\in 	\YY\subseteq \YY_A$, we see that
$\NN_{\YY_A}(Y^*)\subseteq\NN_{\YY}(Y^*)$. This together with 
\cref{eq:DNN_Db} shows that \cref{eq:DNNb} holds with $(Y^*, Z^*)$ in the place of $(Y,Z)$. 
Thus, we have shown that $(R^*,Y^*, Z^*)$ also 
solves~\Cref{eq:optcond}.
\end{proof}

\section{The \rPRSM Algorithm }
\label{sect:PRSM}
We now present the details of a modification of 
the so-called restricted contractive Peaceman-Rachford splitting method,
\PRSMp, or symmetric \ADMMp, e.g.,~\cite{doi:10.1137/13090849X,MR3359677}. 
Our modification involves redundant constraints on subproblems as well
as on the update of dual variables.

\subsection{Outline and Convergence for \rPRSMp}
The augmented Lagrangian function for \Cref{eq:main} with 
Lagrange multiplier $Z$ is: 
\index{$L_Q \leftarrow 
\widehat{V} \widehat{V}^T   L_Q \widehat{V} \widehat{V}^T$}
\index{augmented Lagrangian, $\LL_A$}
\index{$\LL_A$, augmented Lagrangian}
\begin{equation}
\label{eq:augmented_lagrangian}
\LL_A (R, Y,Z) = \langle L_{Q}, Y \rangle  + \langle Z,Y - {\hV} R 
{\hV}^T \rangle + \frac{\beta}{2} \norm{Y - {\hV} R \widehat 
V^T}^{2}_F,
\end{equation}
where $\beta$ is a positive penalty parameter.

\index{$\cP_{\ZZ_0}$, projection onto $\ZZ_{0}$}
\index{projection onto $\ZZ_{0}$, $\cP_{\ZZ_0}$}
Define $\ZZ_{0} := \{Z\in\mathbb{S}^{n^2+1}: Z_{i,i} =0, 
\ Z_{0,i}=Z_{i,0}= 0 , \ i=1,\ldots ,n^2 \}$ and let
$\cP_{\ZZ_0}$ be the projection onto the set $\ZZ_{0}$. 
Our proposed algorithm reads as follows:

%\index{$\NN_{S}(x)$, normal cone of $S$ at $x$} 
%\index{normal cone of $S$ at $x$, $\NN_{S}(x)$} 
%\index{under-relaxation parameter, $\gamma\in(0,1)$} 
%\index{penalty parameter, $\beta\in(0,\infty)$}
%
%\begin{theorem}[characterization of optimality for
%\cref{eq:augmented_lagrangian}]
%The primal-dual $R,Y,Z$ are optimal for  \cref{eq:augmented_lagrangian} 
%if, and only if,
%\begin{subequations}\label{eq:optcondnewaugL}
%	\begin{alignat}{4}
%	R &=  \mathcal{P}_\RR (R+\widehat{V}^TZ\widehat{V})
%\\	Y& = \mathcal{P}_\YY (Y- L_Q - Z) 
%\\	Y & =   \widehat{V}R\widehat{V}^T.
%	\end{alignat}
%\end{subequations} 
%Moreover, at optimality we can restrict  the  optimal $Z$ using
%$Z =  \mathcal{P}_{{\mathcal Z}_A}(Z-(Y - \hV R \hV^T))$.
%\end{theorem}
%\begin{proof}
%\begin{noteH}
%?????????????  add proof????????
%\end{noteH}
%\end{proof}

\begin{algorithm}[h!]  
\caption{\rPRSM for \DNN in \cref{eq:main}}
\label{algo:PRS_algorithm} 
\begin{algorithmic}
\STATE \textbf{Initialize:} 
$\LL_A$ augmented Lagrangian in \Cref{eq:augmented_lagrangian}; 
\textdef{$\gamma \in(0,1)$, under-relaxation parameter}; 
\textdef{$\beta\in (0,\infty)$, penalty parameter};
$\RR$,$\YY$ from \Cref{def:R}; $Y^0$; and $Z^0\in \ZZ_A$; 
\WHILE {tolerances not met}
\STATE $R^{k+1} = \argmin\limits_{R \in\RR} \LL_A (R,Y^k,Z^k)$
\STATE $Z^{k+\frac{1}{2}} = Z^k + \gamma  \beta \cdot \cP_{\ZZ_0}\left( 
Y^{k} - \widehat{V} R^{k+1} \widehat{V}^T \right)$
\STATE $Y^{k+1} = \argmin\limits_{Y\in \YY} \LL_A (R^{k+1},Y,Z^{k+\frac{1}{2}}) $
\STATE $Z^{k+1} = Z^{k+\frac{1}{2}} + \gamma  \beta \cdot 
\cP_{\ZZ_0}\left( Y^{k+1}- \widehat{V} R^{k+1} \widehat{V}^T \right)$ 
\ENDWHILE
\end{algorithmic}
\end{algorithm}
\noindent
\begin{remark}
\Cref{algo:PRS_algorithm} can be summarized as
follows: alternate minimization of variables $R$  and $Y$ interlaced by 
the dual variable $Z$ update.
Before discussing the convergence of \Cref{algo:PRS_algorithm}, we
point out the following.
The $R$-update and the $Y$-update in
\Cref{algo:PRS_algorithm} are well-defined, i.e.,~the subproblems 
involved have unique solutions. 
This follows from the strong convexity of $\LL_A$ with respect to 
$R,Y$ and the convexity and compactness of the sets $\RR$ and $\YY$.
We also note that, in \Cref{algo:PRS_algorithm}, we update the dual 
variable $Z$ both after the $R$-update and the $Y$-update. 

This pattern of update in our Algorithm~\ref{algo:PRS_algorithm} is 
closely related to the strictly contractive \textdef{Peaceman-Rachford 
splitting 
method, 
\PRSMp}; see 
e.g.,~\cite{doi:10.1137/13090849X,MR3359677}. Indeed, we show in 
\Cref{thm:seqcnvg} below, that our algorithm can be 
viewed as a version of \textit{semi-proximal strictly contractive
\PRSMp}, see e.g.,~\cite{2015arXiv150602221G,MR3359677}, 
applied to \Cref{eq:DNNfinalre-simple}. 
Hence, the convergence of our
algorithm can be deduced from the general convergence theory of
semi-proximal strictly contractive \PRSMp.
\index{semi-proximal strictly contractive \PRSMp, semi-proximal \rPRSMp}
\index{\PRSMp, Peaceman-Rachford splitting method}
\index{\rPRSMp, restricted Peaceman-Rachford splitting method}
\end{remark}
\begin{theorem}
\label{thm:seqcnvg}
Let $\{R^k\}, \{Y^k\}, \{Z^k\}$ be the sequences generated by
Algorithm~\ref{algo:PRS_algorithm}. Then the sequence $\{(R^k,Y^k)\}$ converges to a
primal optimal pair $(R^*,Y^*)$ of~\Cref{eq:main}, and $\{Z^k\}$ 
converges to an optimal dual solution $Z^*\in {\mathcal Z}_A$.
\end{theorem} 
\begin{proof}
%Let $\PP$ be a linear map. Then we note that
%\[
%\begin{array}{c}
%Y = \widehat{V}R\widehat{V}^{T}
%\\
% \iff
%\\
%	\PP (Y-\widehat{V}R\widehat{V}^{T})=0, \,
%	(I-\PP) (Y-\widehat{V}R\widehat{V}^{T})=0.
%\end{array}
%\]
The proof is divided into two steps. In the first step,
we consider the convergence of the semi-proximal restricted 
contractive \PRSM in~\cite{2015arXiv150602221G, 
	MR3359677}
applied to the following problem 
\Cref{eq:DNNfinalre-simple}, where $\mathcal{P}_{\ZZ_0^c}$ is the projection onto the orthogonal complement of $\ZZ_{0}$, i.e., $ \mathcal{P}_{\ZZ_0^c} = I-\mathcal{P}_{\ZZ_0}$:
	\begin{equation}\label{eq:DNNfinalre-simple}
	\begin{array}{rcl}
	&\min\limits_{R,Y} & \langle L_Q, 
	\PP_{\ZZ_0}(Y)+\PP_{\ZZ_0^c}(VRV^T)\rangle \\
	&{\rm s.t.} & \PP_{\ZZ_0}(Y)= 
	\PP_{\ZZ_0} (\widehat{V}R\widehat{V}^{T}) \\
	&    & R \in \RR\\
	&    & Y \in \YY.
	\end{array}
	\end{equation}	
We show that the sequence generated by the semi-proximal restricted 
contractive \PRSM in 
	\cite{2015arXiv150602221G,MR3359677} converges to a
\textdef{Karush-Kuhn-Tucker, \KKT} point of \Cref{eq:main}.
In the second step,
we show that the sequence generated by Algorithm~\ref{algo:PRS_algorithm} 
is identical with the sequence generated by the semi-proximal restricted 
contractive \PRSM applied 
to \Cref{eq:DNNfinalre-simple}. 
\index{\KKTp, Karush-Kuhn-Tucker}

%The convergence of the sequence generated by the semi-proximal
%restricted contractive  \rPRSM 
%follows from the general theory presented in~\cite{2015arXiv150602221G, 
%	MR3359677}.
\paragraph{Step 1:} 
We apply the semi-proximal strictly contractive \PRSM given in 
\cite{2015arXiv150602221G,MR3359677} 
to~\Cref{eq:DNNfinalre-simple}.
Let $(\tilde R^0,\tilde Y^0,\tilde Z^0):= (R^0,Y^0,Z^0)$, where $R^0$ and 
$Y^0$ are chosen to satisfy~\Cref{eq:main} and $Z^0\in {\mathcal Z}_A$. 
Consider the following update:
{\small
	\begin{equation}\label{eq:semi-prox-PRSM}  
\hspace{-0.1cm}  %this hspace is to place the label on the right side properly	
\begin{array}{l}
	\tilde R^{k+1}  =
	\argmin\limits\limits_{R\in \RR}  \langle L_Q, 
	\PP_{\ZZ_0^c}(\widehat{V}R\widehat V^T)\rangle\!-\! 
	\langle \tilde Z^k, 
	\mathcal{P}_{\ZZ_0}(\widehat{V}R \widehat{V}^T) \rangle \!+\! 	
	\frac{\beta}{2} \norm{ \mathcal{P}_{\ZZ_0}\!(\tilde 
	Y^k - \widehat{V}R \widehat{V}^T\!)}_F^2\!\!\!+\!\!
 \frac{\beta}{2}\norm{\mathcal{P}_{\ZZ_0^c}(\widehat{V} 
 R\widehat{V}^T \!-\!\widehat{V}\tilde R^k\widehat{V}^T)}_F^2,\\
	\tilde Z^{k+\frac12} = \tilde Z^k + \gamma 
	\beta\mathcal{P}_{\ZZ_0}(\tilde Y^k - \widehat{V} \tilde R^{k+1} 
	\widehat{V}^T ),\\
	\tilde Y^{k+1}  \in
	\argmin\limits\limits_{Y\in\YY} \ \langle L_Q,\PP_{\ZZ_0}(Y)\rangle  + 
	\langle 
	\tilde 
	Z^{k+\frac12}, 
	\mathcal{P}_{\ZZ_0}(Y) \rangle +
	\frac{\beta}{2} \norm{\mathcal{P}_{\ZZ_0}( Y-\widehat{V}\tilde 
	R^{k+1} \widehat{V}^T)}_F^2,\\
	\tilde Z^{k+1} = \tilde Z^{k+\frac12} + \gamma \beta 
	\mathcal{P}_{\ZZ_0}(\tilde Y^{k+1} - \widehat{V} \tilde R^{k+1} 
	\widehat{V}^T),
	\end{array}
	\end{equation}}
\!\!where $\gamma \in (0,1)$ is an under-relaxation parameter.
	Note that the $R$-update in~\Cref{eq:semi-prox-PRSM} is well-defined 
	because	the subproblem involved is a strongly convex 
	problem. By completing the square in the $Y$-subproblem, we have that 
	\[
\tilde Y^{k+1}  
\in \argmin\limits\limits_{Y\in\YY}  
\left\|\cP_{\ZZ_0} (Y) 
- 
  \left(  \cP_{\ZZ_0} ( 
  \widehat{V} \tilde R^{k+1} \widehat{V}^T ) -\frac{1}{\beta} ( L_Q +\tilde 
  Z^{k+\frac12})  \right) \right\|_F^2.
  \]
  We note that ${\cal P}_{{\mathcal Z}_0}(\tilde Y^{k+1})$ is uniquely 
  determined with 
  \[
  {\cal P}_{{\mathcal Z}_0}(\tilde Y^{k+1})
  = \cP_{\ZZ_0} (  \widehat{V} \tilde R^{k+1} \widehat{V}^T ) 
  -\frac{1}{\beta} ( L_Q +\tilde Z^{k+\frac12}) ,
  \]
  while ${\cal P}_{{\mathcal Z}_0^c}(\tilde Y^{k+1})$ can be chosen to be 
  \begin{equation}\label{eq:def-Y}
  \mathcal{P}_{\ZZ^c_0}(\tilde Y^{k+1})
  = \mathcal{P}_{\ZZ_0^c}(\widehat{V}\tilde 
  R^{k+1}\widehat{V}^T) \ , \ \ \ \ \forall\, k\ge 0.
  \end{equation}
  Finally, one can also deduce by induction that $\tilde Z^k\in {\mathcal 
  Z}_A$, for all $k$, since $Z^0\in {\mathcal Z}_A$.
  From the general convergence theory of semi-proximal strictly 
  contractive \PRSM given in~\cite{2015arXiv150602221G,MR3359677}, we have 
  \[
  \left( \tilde R^k, \ \tilde Y^k, \ \tilde Z^k \right)
  \to \Big( R^*  , Y^*,  Z^* \Big) \in \RR \times \YY \times 
  {\mathcal Z}_A,
  \]
  where the convergence of $\{\tilde{R}^k\}$ follows from the injectivity 
  of the map $R\mapsto \widehat{V}R\widehat{V}^T$. Thus, the 
  triple
  $(R^*,Y^*,Z^*)$  solves the optimality condition for 
  \Cref{eq:DNNfinalre-simple}, ~i.e., 
  \begin{subequations}
  \label{eq:temp_oc}
  \begin{alignat}{4}
  &0  \in \widehat V^T\PP_{\ZZ_0^c}(L_Q)\widehat V -\widehat V^T
  \PP_{\ZZ_0}(Z^*)\widehat V+\NN_{\RR}(R^*)
  	\label{eq:temp_oca} \\
  	&0 \in \PP_{\ZZ_0}(L_Q)+\PP_{\ZZ_0}(Z^*)+\NN_{\YY}(Y^*)
  \label{eq:temp_ocb} \\
  &	\PP_{\ZZ_0}(Y^*) = \PP_{\ZZ_0}(\widehat VR^*\widehat V^T). 
  \label{eq:temp_occ}
  \end{alignat}
  \end{subequations}
  Since we update $\PP_{\ZZ_0^c}(\tilde{Y}^k)$ by \Cref{eq:def-Y}, we also have that 
  \begin{equation}\label{eq:temp_ocd}
  \PP_{\ZZ_0^c}(Y^*)=\PP_{\ZZ_0^c}(\widehat V R^*\widehat V^T). 
  \end{equation}
  Next we show that the triple $(R^*,Y^*,Z^*)$ is also a \KKT point of 
  model 
  \Cref{eq:main}. Firstly, It follows from \Cref{eq:temp_occ} and 
  \Cref{eq:temp_ocd}
  that 
  \[
  Y^*=\widehat V R^*\widehat V^T.
  \]
  Secondly, we can deduce from \Cref{eq:temp_oca}, \Cref{eq:temp_ocb} and 
  $Z^*\in\ZZ_A$ that 
  \[
  0\in -\widehat V^TZ^*\widehat V+\NN_\RR(R^*)\quad\hbox{and}\quad
  0\in L_Q+Z^*+\NN_\YY(Y^*).
  \]
  Hence, we have shown that the sequence generated by 
  by~\Cref{eq:semi-prox-PRSM} and~\Cref{eq:def-Y},  converges to a \KKT 
  point 
  of the model \Cref{eq:main}.

\paragraph{Step 2:} 
We now claim that the sequence $\{(\tilde R^k,\tilde 
Z^{k-\frac12},\tilde{Y}^k,\tilde Z^k)\}$ generated 
by~\Cref{eq:semi-prox-PRSM} and~\Cref{eq:def-Y}, starting 	from $(\tilde 
R^0,\tilde Y^0,\tilde Z^0):= (R^0,Y^0,Z^0)$, is identical to the sequence 
$\{(R^k,Z^{k-\frac12},Y^k,Z^k)\}$ given by 	
Algorithm~\ref{algo:PRS_algorithm}.
We prove by induction. 
First, we clearly have $(\tilde R^0,\tilde Y^0,\tilde Z^0)= (R^0,Y^0,Z^0)$ by the definition. 
Suppose that $(\tilde R^k,\tilde Y^k,\tilde Z^k)= (R^k,Y^k,Z^k)$ for some $k\ge 0$.
Since $\tilde{Z}^k \in \ZZ_A$ and~\Cref{eq:def-Y} holds, we can rewrite 
the $R$-subproblem in~\Cref{eq:semi-prox-PRSM} as follows:
{\small	\[
\begin{array}{rl}
&\argmin\limits\limits_{R\in\RR} \langle L_Q, 
\PP_{\ZZ_0^c}(\widehat{V}R\widehat V^T)\rangle- 
\langle \tilde Z^k\!, 
\mathcal{P}_{\ZZ_0}\!(\widehat{V}R \widehat{V}^T\!) \rangle \! + \!	
\frac{\beta}{2} \norm{ \mathcal{P}_{\ZZ_0}\!(\tilde 
	Y^k\! - \!\widehat{V}R \widehat{V}^T\!)}_F^2\!+
\frac{\beta}{2}\norm{\mathcal{P}_{\ZZ_0^c}(\widehat{V}\tilde 
	R^k\widehat{V}^T \!-\!\widehat{V}R\widehat{V}^T)}_F^2\\
=&\argmin\limits\limits_{R\in\RR}\langle \PP_{\ZZ_0^c}(L_Q)-\PP_{\ZZ_0}(\tilde 
Z^{k}),\widehat 
VR\widehat V^T 
\rangle+\!	
\frac{\beta}{2} \norm{ \mathcal{P}_{\ZZ_0}\!(\tilde 
	Y^k\! - \!\widehat{V}R \widehat{V}^T\!)}_F^2 \! \!
+\!
\frac{\beta}{2}\norm{\mathcal{P}_{\ZZ_0^c}(\widehat{V}\tilde 
	R^k\widehat{V}^T \!-\!\widehat{V}R\widehat{V}^T)}_F^2\\
=&\argmin\limits\limits_{R\in\RR}\langle -\PP_{\ZZ_0^c}(\tilde 
Z^k)  - 
\PP_{\ZZ_0}(\tilde Z^k), 
\widehat{V}R 
\widehat{V}^T 
\rangle +\frac{\beta}{2} \norm{ \tilde Y^k-\widehat{V}R 
	\widehat{V}^T}_F ^2\\
=&\argmin\limits\limits_{R\in\RR}   - \langle \tilde Z^k, \widehat{V}R 
\widehat{V}^T 
	\rangle +\frac{\beta}{2} \norm{ \tilde Y^k-\widehat{V}R 
	\widehat{V}^T}_F ^2,
\end{array}
\]}where the second ``$=$'' is due to $\tilde Z^k\in\ZZ_A$ and \Cref{eq:def-Y}.
The above is equivalent to the $R$-subproblem in 
\Cref{algo:PRS_algorithm}, since $\tilde Z^k = Z^k$ and $\tilde Y^k = Y^k$ 
by the induction hypothesis. 
This shows that $\tilde R^{k+1} = R^{k+1}$ and it follows that $\tilde 
Z^{k+\frac12} = Z^{k+\frac12}$.
Since $Z^{k+\frac12}\in {\mathcal Z}_A$,
we can rewrite the $Y$-subproblem in \Cref{algo:PRS_algorithm} as
{\small\[
\hspace{-0.2cm}\begin{array}{cl}
	&\argmin\limits\limits_{Y\in\YY} \langle L_Q+ 
	Z^{k+\frac12},Y\rangle + \frac\beta{2}\|Y - 
	{\hV}	R^{k+1}{\hV}^T\|_F^2 \\
	=&\argmin\limits\limits_{Y\in\YY} \langle 
	\PP_{\ZZ_0}(L_Q+Z^{k+\frac12}), Y\rangle+\frac\beta{2}\|{\cal 
	P}_{{\mathcal 
		Z}_0}(Y - 
{\hV}	R^{k+1}{\hV}^T)\|_F^2 + \frac\beta{2}\|{\cal P}_{{\mathcal 
		Z}_0^c}(Y - {\hV} R^{k+1}{\hV}^T)\|_F^2 \\
	=&\argmin\limits\limits_{Y\in\YY}\langle 
	L_Q,\PP_{\ZZ_0}(Y)\rangle  + 
	\langle  
	Z^{k+\frac12},
	\mathcal{P}_{\ZZ_0}(Y) \rangle +
	\frac{\beta}{2} \norm{\mathcal{P}_{\ZZ_0}( Y-\widehat{V} 
		R^{k+1} \widehat{V}^T)}_F^2+ \frac\beta{2}\|{\cal P}_{{\mathcal 
			Z}_0^c}(Y - {\hV} R^{k+1}{\hV}^T)\|_F^2,\\
 \end{array}
\]}where the first ``$=$'' is due to $Z^{k+\frac12}\in\ZZ_A$.
Hence, with $\tilde R^{k+1} = R^{k+1}$ and $\tilde 	Z^{k+\frac12} = 
Z^{k+\frac12}$, we have that the above subproblem generates 
$\tilde{Y}^{k+1}$
defined in~\Cref{eq:semi-prox-PRSM} and~\Cref{eq:def-Y}. 
Thus we have $\tilde Y^{k+1} = Y^{k+1}$ and it follows that $\tilde Z^{k+1} = Z^{k+1}$ holds. 
This completes the proof for 
$\{(R^k,Y^k,Z^k)\}_{k\in \mathbb{N}} \equiv \{(\tilde R^k,\tilde 	
Y^k,\tilde Z^k)\}_{k\in \mathbb{N}}$, and the alleged convergence behavior 
of $\{(R^k,Y^k,Z^k)\}$ follows from that of $\{(\tilde R^k,\tilde 
Y^k,\tilde Z^k)\}$. 
\end{proof}

\subsection{Implementation details}

Note that the explicit $Z$-updates in \Cref{algo:PRS_algorithm} is simple 
and easy.
We now show that we have explicit expressions for $R$-updates and 
$Y$-updates as well.

\subsubsection{\mathinhead{R}{R}-subproblem}\label{sec:R}

In this section we present the formula for solving the $R$-subproblem in 
\Cref{algo:PRS_algorithm}.
We define $\mathcal{P}_{\RR}(W)$ to be the projection of $W$ onto the compact set $\RR$, where $\RR := \left\{R\in \mathbb{S}_+^{(n-1)^2+1} :  \trace  (R)  = n+1 \right\}$.
By completing the square at the current iterates $Y^{k},Z^{k}$,
the $R$-subproblem can be explicitly solved by the projection operator $\mathcal{P}_{\RR}$ as follows:
\[
\begin{array}{cl}
R^{k+1}
& = \argmin\limits\limits_{R\in\RR}
-\langle Z^k,{\hV}R{\hV}^T \rangle
+ \frac {\beta}2   \norm{ Y^k-\widehat{V}R\widehat{V}^T}_F^2
  \\ & = \argmin\limits\limits_{R\in\RR} \frac{\beta}{2}\norm{
  Y^k-\widehat{V}R\widehat{V}^T + \frac{1}{\beta} Z^k}_F^2  \\
  & = \argmin\limits\limits_{R\in\RR } \frac{\beta}{2}\norm{ R- \widehat{V}^T
  ( Y^k+\frac{1}{\beta}  Z^k  ) \widehat{V} }_F^2   \\
  & =\mathcal{P}_{\RR}(\widehat{V}^T(Y^k+\frac{1}{\beta} Z^k) \widehat{V}),
  \end{array}
  \]
  where the third equality follows from  the assumption 
  $\widehat{V}^T\widehat{V}=I$.

For a given symmetric matrix $W \in \mathbb{S}^{(n-1)^2+1}$, we now show how to perform the projection $\mathcal{P}_{\RR}(W) $.
Using the eigenvalue decomposition $W = U \Lambda U^T$, 
we have 
\begin{equation*}
\mathcal{P}_{\RR}(W)=U\Diag(\mathcal{P}_{\Delta}(\diag(\Lambda)))U^T,
\end{equation*}
where $\mathcal{P}_{\Delta}(\diag(\Lambda))$ denotes the projection of $\diag(\Lambda)$ onto the simplex 
\[
\Delta =\left\{\lambda\in \mathbb{R}^{(n-1)^2+1}_+ : \lambda^Te=n+1 \right\}.
\]
Projections onto simplices can be performed efficiently via some standard 
root-finding strategies; see, for example~\cite{par,chen2011projectionOA}.
Therefore the $R$-updates reduce to the projection of the vector 
of the  positive eigenvalues of $ {\hV}^T \left( Y^k + \frac{1}{\beta} Z^k 
\right) {\hV}$ onto the simplex $\Delta$.

\subsubsection{\mathinhead{Y}{Y}-subproblem}\label{sec:Y}
In this section we present the formula for solving the $Y$-subproblem in 
\Cref{algo:PRS_algorithm}. By completing the square at the
current iterates $R^{k+1},Z^{k+\frac 12}$, we get 
\begin{equation*}
\label{eq:Yupdate}
\begin{array}{cl}
Y^{k+1} & =\argmin\limits\limits_{Y\in\YY} \langle L_Q, Y\rangle+\langle  
Z^{k+\frac{1}{2}},
Y-\widehat{V} R^{k+1}\widehat{V}^T\rangle
+\frac{\beta}{2}\norm{Y-\widehat{V}R^{k+1}\widehat{V}^T}_F^2\\
&= \argmin\limits\limits_{Y\in\YY} \frac{\beta}{2}\norm{ Y- 
\left(\widehat{V}R^{k+1}
	\widehat{V}^T - \frac{1}{\beta}( L_Q+ Z^{k+\frac{1}{2}})\right) }_F^2.
\end{array}
\end{equation*}
Hence the $Y$-subproblem involves the projection onto the polyhedral set 
\index{polyhedral constraints, $\YY$}
\index{$\YY$, polyhedral constraints}
\[
\YY:=\{Y\in\mathbb{S}^{n^2+1}:\GG_{\bar J}(Y) = u_0, 0\leq Y\leq 1\}.
\]
Let
$T := \left(\widehat{V}R^{k+1} \widehat{V}^T 
           - \frac{1}{\beta}( L_Q+ Z^{k+\frac{1}{2}})\right)$.
Then we update $Y^{k+1}$ as follows:
\[
(Y^{k+1})_{ij}=
\left\{
\begin{array}{ll}
1 & \text{if  } i=j=0,\\
0 & \text{if  } ij\;\hbox{or }\;ji\in \bar{J}/{(00)},\\
\min\left\{ 1,\max\{(\widehat V R^{k+1}\widehat V^T)_{ij},0\} \right\} & 
\text{if  } i=j\neq0\;\hbox{or 
}\;ij=0\neq i+j,\\
\min \left\{ 1,\max\{T_{ij},0\} \right\}  & \hbox{otherwise.}
\end{array}
\right.
\]

\subsection{Bounding}
\label{sec:bounding}

\index{global optimal value of \QAPp, $p^*_{\QAPp}$}
\index{$p^*_{\QAPp}$, global optimal value of \QAPp}
In this section we present some strategies for obtaining lower and upper bounds for $p^*_{\QAPp}$.

\subsubsection{Lower Bound from Relaxation}
\label{sec:lowerbound}

Exact solutions of the relaxation \Cref{eq:main} provide lower bounds to 
the original \QAP \Cref{eq:qap}.
However, the size of problem \Cref{eq:main} can be extremely large, and it 
could be very expensive to obtain solutions of high accuracy.
In this section we present an inexpensive way to obtain a valid lower bound
using the output with moderate accuracy from our algorithm.

\index{$g$, dual functional}
\index{dual functional, $g$}
Our approach is based on the following functional
\begin{equation}\label{eq:g}
g(Z):=\min_{Y\in\YY}\left\langle L_Q+Z ,Y\right\rangle -(n+1) 
\lambda_{\max}({\hV}^TZ{\hV}),
\end{equation}
where \textdef{$\lambda_{\max}(\widehat{V}^TZ\widehat{V})$}
denotes the largest eigenvalue of $\widehat{V}^T Z \widehat{V}$.

In  \Cref{thm:lowerbnd} below, we show that $\max\limits_Z g(Z)$ is indeed 
the Lagrange dual problem of our main problem \Cref{eq:main}. 
\begin{theorem}
\label{thm:lowerbnd}
Let $g$ be the functional defined in \Cref{eq:g}.
Then the problem 
\begin{equation}\label{eq:Fenchel_dual}
d^*_Z:=\max_{Z}g(Z)
\end{equation}
is a concave maximization problem.
Furthermore, strong duality holds for the problem  \Cref{eq:main} and  
\Cref{eq:Fenchel_dual}, i.e.,
\[
p^*_{\bf DNN} = d^*_Z,\ \mbox{and $d^*_Z$ is attained.}
\]
\end{theorem}
\begin{proof}
Note that the function $\widehat{V}^T Z \widehat{V}$ is linear in $Z$. 
Therefore
the largest eigenvalue function 
$\lambda_{\max}(\widehat{V}^TZ\widehat{V})$ is a convex function of $Z$.
Thus the function
\[
\left\langle L_Q+Z ,Y\right\rangle -(n+1) 
\lambda_{\max}({\hV}^TZ{\hV})
\]
is concave in $Z$. The concavity of $g$ is now clear.

We derive \Cref{eq:Fenchel_dual} via the Lagrange dual problem of 
\Cref{eq:main}:
\begin{subequations}
\label{eq:weakdualsZ}
\begin{alignat}{4}
	p^*_{\bf DNN}
     &=\min\limits_{R\in\RR,Y\in \YY}
		\max\limits_Z
		\left\{\left\langle L_Q,Y\right\rangle+
		\langle Z,Y-\widehat{V}R\widehat{V}^T\rangle \right\}
	     \nonumber
		\\ &=
		\label{eq:weakdualZb}
		\max\limits_Z\min\limits_{R\in\RR,Y\in\YY}
		\left\{\left\langle L_Q,Y\right\rangle+
		\langle Z,Y-\widehat{V}R\widehat{V}^T\rangle \right\}
		\\ &=
		\nonumber
		\max\limits_Z\left\{\min\limits_{Y\in\YY}\left\{\left\langle
		L_Q,Y\right\rangle+\langle Z,
		Y\rangle\right\}+\min_{R\in\RR}\langle 
		Z,-\widehat{V}R\widehat{V}^T\rangle\right\}\\
		&=\max\limits_Z\left\{\min\limits_{Y\in\YY}\left\{\left\langle
		\nonumber
		L_Q,Y\right\rangle+\langle
		Z,Y\rangle\right\}+\min_{R\in\RR}\langle
		\widehat{V}^TZ\widehat{V},-R\rangle\right\}
		\\ &=
		\label{eq:weakdualZc}
		\max\limits_{Z}\left\{\min\limits_{Y\in\YY}\left\langle 
		L_Q+Z,Y\right\rangle-(n+1)\lambda_{\max}(\widehat{V}^TZ\widehat{V})\right\} \\
		& = \nonumber d_Z^*,
		\end{alignat}
	\end{subequations}
	where:
	\vspace{-.5em}
	\begin{enumerate}
		\item
		That \Cref{eq:weakdualZb} follows 
		from~\cite[Corollary~28.2.2, Theorem~28.4]{Rocka:70}
and the fact that \Cref{eq:main} has generalized Slater 
points, see~\cite{KaReWoZh:94}.\footnote{Note that the Lagrangian 
is linear in $R,Y$ and linear in $Z$.
			Moreover, both constraint sets $\RR,\YY$ are convex and 
			compact. Therefore, the
			result also follows from the classical Von Neumann-Fan minmax 
			theorem.}
		\item
		That \Cref{eq:weakdualZc} follows from the definition of 
		$\RR$ and the
		Rayleigh Principle.
	\end{enumerate}
We see from~\cite[Corollary~28.2.2, Corollary~28.4.1]{Rocka:70}  that the 
dual optimal value $d^*_Z$ is attained.
\end{proof}

Since the Lagrange dual problem in \Cref{thm:lowerbnd}
is an unconstrained maximization problem,
evaluating $g$ defined in \Cref{eq:g} at the $k$-th iterate $Z^k$
yields a valid lower bound for $p^*_{\DNNp}$, i.e.,~$g(Z^k) \le 
p^*_{\DNNp}\le p^*_{\QAPp}$.
The functional $g$ also strengthens the bound given in 
\cite[Lemma~3.2]{OliveiraWolkXu:15}.
We also see in \Cref{eq:weakdualZc} that $Z\prec 0$
provides a positive contribution to the eigenvalue part of the lower
bound. Moreover, \Cref{thm:LQzero} implies that the contribution from
the diagonal, first row and column of $L_Q+Z$ (except for the $(0,0)$-th 
element) is zero. This motivates scaling $L_Q$ to be positive definite.
Let \textdef{$P_V$}$:= {\hV} {\hV}^T$.
Then for any $r,s\in\mathbb{R}$, the objective in \Cref{eq:main} 
can be replaced by 
\begin{equation}
\label{eq:scaled-version}
\langle r(P_VL_QP_V+s I),Y \rangle.
\end{equation}
We obtain the same solution pair $(R^*, Y^*)$ of \Cref{eq:main}.
%by solving the model \Cref{eq:main} with 
%a scaled and shifted 
%$L_Q$, ~i.e., $r(P_VL_QP_V+sI)$. 
Another advantage is that it potentially 
forces the dual multiplier $Z^*$ to be negative definite, and thus
the lower bound is larger.  
\begin{remark}
Additional strategies can be used to strengthen the lower bound $g(Z^k)$.
Suppose that the given data matrices $A,B$ are symmetric and integral, 
then from \Cref{eq:qap}, we know that $p^*_{\QAPp}$ is an even integer.
Therefore applying the ceiling operator to $g(Z^k)$ still gives a valid 
lower bound to $p^*_{\QAPp}$. According to this prior information,  we can 
strengthen the lower bound with the even number in the pair 
$ \{\ceil[\big]{g(Z^k)},\ceil[\big]{g(Z^k)}+1\}$. 
%As mentioned in~\cite{MR2979433}, 

\end{remark}

\subsubsection{Upper Bound from Nearest Permutation Matrix}
\label{sec:upper_bound}

In \cite{OliveiraWolkXu:15}, the authors present two methods for obtaining 
upper bound from nearest permutation matrices. 
In this section we present a new strategy for computing upper bounds from 
nearest permutation matrices.

Given $\bar{X} \in \R^{n\times n}$, the nearest permutation matrix $X^*$ 
from $\bar{X}$ is found by solving
\begin{equation}
\label{prob:nearest_permutation_matrix}
X^* 
\ =  \ \argmin\limits\limits_{X\in \Pi} \frac{1}{2} \|X - \bar{X}\|^2_F
\ = \ \argmin\limits\limits_{X\in \Pi} -\langle \bar{X},X \rangle.
\end{equation}
Any solution to the problem \Cref{prob:nearest_permutation_matrix} yields 
a feasible solution to the original \QAPp, which gives a valid upper bound 
$\trace (A X^* B  (X^*)^T)$.
It is well-known that the set of $n$-by-$n$ permutation matrices is the 
set of extreme points of the set of doubly stochastic matrices $\{ X \in 
\R^{n\times n} : Xe=e, \, X^{T}e=e, \, X \geq 0 \}$.\footnote{It is known 
	as Birkhoff-von Neumann theorem \cite{MR0054920,Birk:46}.}
Hence we reformulate the problem \Cref{prob:nearest_permutation_matrix} as
\begin{equation}
\label{prob:LPversion_nearest_permutation}
\max\limits\limits_{x\in \R^{n^2}} \left\{  \langle  \kvec(\bar{X}), x 
\rangle \ : \  (I_n \otimes e^T) x = e, \ (e^T \otimes I_n)x = e, \ x\ge 0 
\right\}
\end{equation}
and we solve \eqref{prob:LPversion_nearest_permutation} using simplex 
method.
Suppose that we have found an approximate optimum $Y^\text{out}$ for our 
\DNN
relaxation.
The first approach presented in \cite{OliveiraWolkXu:15} is to set 
$\kvec(\bar{X})$ to be  the second through the last elements of the first 
column of $Y^{\text{out}}$ and solve 
\Cref{prob:LPversion_nearest_permutation}.
Now suppose that we further obtain the spectral decomposition of the 
approximate optimum 
\[
Y = \sum_{i=1}^r \lambda_i v_iv_i^T,
\]
with $\lambda_1\ge \lambda_2\ge \cdots \ge \lambda_r>0$.
And by abuse of notation we set $v_i$ to be the vectors in $\R^{n^2}$
formed by removing the first element from $v_i$.
The second approach presented in \cite{OliveiraWolkXu:15} is 
to use $\kvec(\bar{X}) = \lambda_1 v_1$ in solving 
\Cref{prob:LPversion_nearest_permutation}, where $(\lambda_1,v_1)$ is the 
most dominant eigenpair of $Y^\text{out}$.

We now present our new approach inspired by \cite{Goemans}.
Let $\xi$ be a random vector in $\R^r$  with its components in $(0,1)$
and in decreasing order. 
We use $\xi$ to perturb the eigenvalues $\lambda_1,\ldots,\lambda_r$ for 
forming $\bar{X}$ as follows:
\begin{equation*}
\label{eqtn:forming_Xbar_random}
\kvec(\bar{X}) = \sum_{i=1}^r \xi_i \lambda_i v_i.
\end{equation*}
In each time we compute the upper bound, we use this approach 
$3\lceil\log(n)\rceil$ times to obtain a bunch of upper bounds, and then 
choose the best (smallest) as the upper bound.

\section{Numerical Experiments with \rPRSMp}
\label{sect:numeric}

%%%%%%%%%%%%%%%%%%%%%%%%%%%%%%%%%%%%%%
We now present the numerical results for \Cref{algo:PRS_algorithm},
that we denote \rPRSM with the bounding strategies 
discussed in \Cref{sec:bounding}. The parameter setting and stopping 
criteria are introduced in \Cref{sec:para_set} below. 
The numerical experiments are divided into two sections. We use symmetric\footnote{We exclude instances that have asymmetric data matrices.} data from QAPLIP\footnote{\url{http://coral.ise.lehigh.edu/data-sets/qaplib/qaplib-problem-instances-and-solutions/}}.
In \Cref{sec:experment1}  we examine the comparative performance between \rPRSM and \cite[\ADMM]{OliveiraWolkXu:15}.
We aim to show that our proposed \rPRSM shows improvements on convergence 
rates and relative gaps as compared to \cite{OliveiraWolkXu:15}. 
In  \Cref{sec:experment2} we compare the numerical performance of \rPRSM with the two recently proposed methods \cite[C-SDP]{MR3772051} and \cite[F2-RLT2-DA]{MR4037827}, that are based on relaxations of the \QAPp.
%and engage lower and upper bounds. 
% Our comparison results are presented in \Cref{fig:FKScomparison}   and \Cref{fig:DNcomparison}. 
%The superiority of  \rPRSM over these two methods  in terms of the bounds and running time are verified.

\subsection{Parameters Setting and Stopping Criteria}
\label{sec:para_set}

\paragraph{Parameter Setting}
We scale the data $L_Q$ as presented in   % \Cref{prop:scaleLQshift}. 
\Cref{eq:scaled-version} as follows:
\[
\begin{array}{ll}
L_1 \leftarrow  P_VL_QP_V ,  &  \\
L_2 \leftarrow  L_1 + \sigma_L  I , & \text{ where } \sigma_L : = \max\{0, -\left\lfloor \lambda_{\min} (L_Q) \right\rfloor\}+10n , \\
L_3 \leftarrow  \frac {n^2}{\alpha}L_2, & \text{ where } \alpha : = \left\lceil \|L_2\|_F\right\rceil .
\end{array}
\]
We set the penalty parameter $\beta = \frac{n}{3}$ and the
under-relaxation parameter $\gamma = 0.9$ for the dual variable update. 
We choose
\[
Y^0 = \frac{1}{n!}\sum_{X\in \Pi} (1;\kvec(X))(1;\kvec(X))^T \ \text{ and } \ Z^0 = \mathcal{P}_{\ZZ_A}(0)
\]
to be the initial iterates\footnote{The formula for $Y^0$ is introduced in \cite[Theorem 3.1]{KaReWoZh:94}.} for \rPRSMp.
We compute the lower and upper bounds every 100 iterations.

%\subsection{Termination Criteria for \rPRSM}
%\label{sec:termination_criteria}
\paragraph{Stopping Criteria}
We terminate \rPRSM when either of the following conditions is satisfied. 
\begin{enumerate}
	\item 
	Maximum number of iterations, denoted by ``maxiter'' is achieved. 
	We set  $\text{maxiter} = 40000$.
	\item 
	For given tolerance $\epsilon$, the following bound on the primal and 
	dual residuals holds for $m_t$ sequential times:
	\[
	\max \left\{ \frac{\|Y^k- \widehat{V}R^k \widehat{V}^T\|_F}{\|Y^k\|_F} 
	, \beta \|Y^{k}-Y^{k-1}\|_F \right\} <\epsilon.
	\]
	We set $\epsilon = 10^{-5}$ and $m_t = 100$.
	\item Let $\{l_1,\ldots,l_k\}$ and $\{u_1,\ldots,u_k\}$ be sequences 
	of 
	lower 
	and upper bounds from \Cref{sec:lowerbound} and 
	\Cref{sec:upper_bound}, 
	respectively.  The lower (resp. upper) bounds do not change for $m_l$ 
	(resp. $m_u$) 
	sequential times. We 
	set 
	$m_l = m_u = 100$.
	\item The
	\KKT conditions given in \Cref{eq:optcondnew} are satisfied to a 
	certain 
	precision. 
	More specifically, for a predefined tolerance $\delta>0$, it holds 
	that 
	\[
	\max \left\{ 
	\| R^k -\mathcal{P}_\RR (R^k +\widehat{V}^TZ^k\widehat{V}) \|_F    , \ 
	\| Y^k\! - \mathcal{P}_\YY (Y^k\!- L_Q - Z^k) \|_F, \ 
	\| Y^k\! - \hV R^k\hV^T \|_F
	\right\} < \delta.
	\] 
	We use this stopping criterion for instances with $n$ larger than $20$ and we set the tolerance $\delta=10^{-5}$ when it is used.
\end{enumerate}
%We include the approaches presented in \cite{OliveiraWolkXu:15} for the 
%sake of completeness. 

\subsection{Empirical Results}
\label{sec:experment1}

In this section we examine the comparative performance of \rPRSM and \cite[\ADMM]{OliveiraWolkXu:15} by using instances from QAPLIB.
We split the instances into three groups based on sizes: 
\[
n \in  \{ 10,\ldots,20\},  
\{ 21,\ldots,40\},  
\{ 41,\ldots,64\}.  
\]
For each group of specific size, we aim to show that our 
proposed \rPRSM shows improvements on convergence and relative gaps from 
\ADMM in 
\cite{OliveiraWolkXu:15}.
We used the parameters for \ADMM as suggested in \cite{OliveiraWolkXu:15}, i.e., $\beta=n/3, \gamma =1.618$.   
We adopt the same stopping criteria for \ADMM as \rPRSM for a proper comparison.
All instances in \Cref{table:QAPLIB_I,table:QAPLIB_II,table:QAPLIB_III} 
are tested using MATLAB version 2018b on a Dell PowerEdge M630 with two Intel Xeon E5-2637v3 4-core 3.5 GHz (Haswell) with	64 Gigabyte memory.
%cpu135.math (traditional)} 
%The instances in \Cref{table:QAPLIB_III} are tested using a Dell PowerEdge R840 with four Intel Xeon Gold 6230 20-core 2.1 GHz (Cascade Lake) 768 Gigabyte memory.  %cpu149.math.private

Below we give some illustrations for the headers in 
\Cref{table:QAPLIB_I,table:QAPLIB_II,table:QAPLIB_III}.
\vspace{-0.5em}
\begin{enumerate}[itemsep=-1mm]
	\item {\bf problem}: instance name;
	\item {\bf opt}: global optimal value of each instance. If the optimal 
	value is unknown, instance name is marked with the asterisk $^*$;
	\item {\bf lbd}: the lower bound obtained by running \rPRSMp;
	\item {\bf ubd}: the upper bound obtained by running \rPRSMp;
	\item {\bf rel.gap}: relative gap of each instance using \rPRSMp, 
	where 
	\begin{equation}
	\label{def:rel_gap}
	\text{relative gap} := 2\ \frac{\text{best feasible upper 
	bound}-\text{best lower bound}}{\ \text{best feasible upper 
	bound}+\text{best lower bound}+1} ;
	\end{equation}
	\item {\bf rel.gap$^\ADMMp$}: relative gap of each instance using~\cite[\ADMM]{OliveiraWolkXu:15} with the tolerance 	$\epsilon = 10^{-5}$;
	\item {\bf iter}: number of iterations used by \rPRSM with tolerance 
	$\epsilon = 10^{-5}$;
	\item {\bf iter$^\ADMMp$}: number of iterations used by~\cite[\ADMM]{OliveiraWolkXu:15} with tolerance $\epsilon = 10^{-5}$;
	\item {\bf time(sec)}: CPU time (in seconds) used by \rPRSMp.
\end{enumerate}

% Our results are presented in \Cref{table:QAPLIB_I,table:QAPLIB_II,table:QAPLIB_III}. 

\subsubsection{Small Size}
\label{sec:QAPLIB-Small}

\Cref{table:QAPLIB_I} contains results on 45 QAPLIB instances
with sizes $n \in \{10\ldots,20\}$. 
\begin{table}[h!]
	%\tiny
	\scriptsize
	\centering
	\caption{QAPLIB Instances of Small Size}
	\label{table:QAPLIB_I}
	\begin{tabular}{|l|cccccccc|}
		\hline    %\toprule    % from the data file    
		%%Small_Result_May5  Small_Result_ADMM_May5
		\rule{-3pt}{8pt}
		\textbf{problem}&\textbf{opt}& \textbf{lbd} & \textbf{ubd} & 
		\textbf{rel.gap} & \textbf{rel.gap$^{\ADMM}$} & \textbf{iter} & 
		\textbf{iter$^{\ADMM}$} & \textbf{time(sec)}  \\
		\hline
chr12a & 9552 & 9548 & 9552 & 0.04 & 0.02 & 11500 & 11300 & 39.30 \\  
chr12b & 9742 & 9742 & 9742 & \textbf{0} & 0 & 10300 & 10600 & 34.75 \\  
chr12c & 11156 & 11156 & 11156 & \textbf{0} & 0 & 1600 & 1700 & 5.38 \\  
chr15a & 9896 & 9896 & 9896 & \textbf{0} & 0 & 8400 & 8800 & 61.00 \\  
chr15b & 7990 & 7990 & 7990 & \textbf{0} & 0 & 4300 & 4000 & 32.78 \\  
chr15c & 9504 & 9504 & 9504 & \textbf{0} & 0 & 2200 & 2200 & 16.44 \\  
chr18a & 11098 & 11098 & 11098 & \textbf{0} & 0 & 3000 & 2500 & 42.04 \\  
chr18b & 1534 & 1534 & 1724 & \textbf{11.66} & 65.60 & 5947 & 3937 & 95.05 \\  
chr20a & 2192 & 2192 & 2192 & \textbf{0} & 0 & 6100 & 5900 & 133.59 \\  
chr20b & 2298 & 2298 & 2298 & \textbf{0} & 0 & 1900 & 3700 & 47.03 \\  
chr20c & 14142 & 14128 & 14142 & 0.10 & 0.02 & 17000 & 39800 & 365.76 \\  
els19 & 17212548 & 17189708 & 17212548 & 0.13 & 0.02 & 21500 & 26000 & 378.69 \\  
esc16a & 68 & 64 & 76 & \textbf{17.02} & 43.64 & 412 & 1176 & 3.70 \\  
esc16b & 292 & 290 & 292 & \textbf{0.69} & 2.72 & 284 & 424 & 2.56 \\  
esc16c & 160 & 154 & 176 & \textbf{13.29} & 32.52 & 397 & 923 & 3.55 \\  
esc16d & 16 & 14 & 16 & \textbf{12.90} & 92 & 280 & 1785 & 2.51 \\  
esc16e & 28 & 28 & 28 & \textbf{0} & 38.24 & 241 & 2237 & 2.41 \\  
esc16g & 26 & 26 & 36 & \textbf{31.75} & 45.45 & 252 & 3401 & 2.23 \\  
esc16h & 996 & 978 & 1100 & \textbf{11.74} & 24.82 & 1137 & 507 & 9.89 \\  
esc16i & 14 & 12 & 14 & \textbf{14.81} & 83.72 & 1445 & 9593 & 13.86 \\  
esc16j & 8 & 8 & 8 & \textbf{0} & 90.32 & 100 & 3382 & 0.98 \\  
had12 & 1652 & 1652 & 1652 & \textbf{0} & 0 & 300 & 1000 & 0.99 \\  
had14 & 2724 & 2724 & 2724 & \textbf{0} & 0 & 500 & 2100 & 3.28 \\  
had16 & 3720 & 3720 & 3720 & \textbf{0} & 0 & 600 & 2100 & 6.62 \\  
had18 & 5358 & 5358 & 5358 & \textbf{0} & 0 & 1900 & 5800 & 30.86 \\  
had20 & 6922 & 6922 & 6922 & \textbf{0} & 0.12 & 3700 & 9440 & 95.06 \\  
nug12 & 578 & 568 & 642 & \textbf{12.22} & 12.22 & 1361 & 5394 & 4.86 \\  
nug14 & 1014 & 1012 & 1022 & \textbf{0.98} & 1.08 & 2940 & 7533 & 19.05 \\  
nug15 & 1150 & 1142 & 1280 & \textbf{11.39} & 15.74 & 1582 & 6111 & 13.88 \\  
nug16a & 1610 & 1600 & 1610 & \textbf{0.62} & 0.62 & 4160 & 11200 & 43.19 \\  
nug16b & 1240 & 1220 & 1250 & \textbf{2.43} & 24.91 & 2405 & 5982 & 23.44 \\  
nug17 & 1732 & 1708 & 1756 & \textbf{2.77} & 2.77 & 3963 & 10469 & 52.89 \\  
nug18 & 1930 & 1894 & 2160 & 13.12 & 4.94 & 5588 & 10900 & 92.14 \\  
nug20 & 2570 & 2508 & 2680 & \textbf{6.63} & 17.30 & 3735 & 9356 & 99.67 \\  
rou12 & 235528 & 235528 & 235528 & \textbf{0} & 0 & 3700 & 6400 & 13.99 \\  
rou15 & 354210 & 350218 & 360702 & \textbf{2.95} & 4.89 & 1924 & 2313 & 15.99 \\  
rou20 & 725522 & 695182 & 781532 & \textbf{11.69} & 14.93 & 3953 & 3778 & 99.21 \\  
scr12 & 31410 & 31410 & 31410 & \textbf{0} & 24.75 & 400 & 1317 & 1.25 \\  
scr15 & 51140 & 51140 & 51140 & \textbf{0} & 2.67 & 800 & 1195 & 6.54 \\  
scr20 & 110030 & 106804 & 132826 & \textbf{21.72} & 35.54 & 5787 & 27800 & 135.58 \\  
tai10a & 135028 & 135028 & 135028 & \textbf{0} & 0 & 1000 & 1700 & 2.15 \\  
tai12a & 224416 & 224416 & 224416 & \textbf{0} & 0 & 300 & 500 & 0.78 \\  
tai15a & 388214 & 377102 & 403890 & \textbf{6.86} & 9.03 & 1957 & 2050 & 16.58 \\  
tai17a & 491812 & 476526 & 534328 & \textbf{11.44} & 16.25 & 2058 & 3091 & 26.62 \\  
tai20a & 703482 & 671676 & 762166 & \textbf{12.62} & 19.03 & 2114 & 2850 & 55.72 \\  
		\hline    %\bottomrule
	\end{tabular}
\end{table}
Columns \textbf{rel.gap} and \textbf{rel.gap$^\ADMMp$} show the improvements on relative gaps on these instances.
In particular, 40 out of 45 instances show competitive relative gaps compared to \ADMM and these instances are marked with boldface in \Cref{table:QAPLIB_I}.
This is due to the improved upper bounds from the random perturbation 
approach presented in 
\Cref{sec:upper_bound}.
In fact, we now have found provably optimal solutions for the following
\underline{\emph{twenty}} instances:
{\footnotesize
	\[
	\begin{array}{llllllllllll}
	\text{chr12b} & \text{chr12c} & \text{chr15a} & \text{chr15b} & 
	\text{chr15c} & \text{chr18a} &  \text{chr20a} & \text{chr20b} & 
	\text{esc16e} & \text{esc16j} \\
	\text{had12}  & \text{had14} & \text{had16} & \text{had18} & 
	\text{had20} & \text{rou12} & \text{scr12} & \text{scr15} & 
	\text{tai10a} & \text{tai12a} .
	\end{array}
	\] }	
Comparing the column \textbf{iter} and the column \textbf{iter$^\ADMMp$}, 
we see that 37 instances were treated with fewer number of iterations 
using \rPRSM than \ADMMp. It shows that \rPRSM converges much faster than
\ADMM for the small-size QAPLIB instances. 

For \rPRSM alone we observe that most of the instances show good bounds 
with reasonable amount of time. Most of the instances are solved within 
two minutes using the machine described above. 
Our algorithm produces strong lower bounds on these instances, mostly 
within 2 percent of the optimum.

\subsubsection{Medium Size}
\label{sec:QAPLIB-Medium}

\Cref{table:QAPLIB_II} contains results on 29 QAPLIB instances with sizes $n \in \{21,\ldots, 40\}$.
We make similar observations in \Cref{sec:QAPLIB-Small}.
\begin{table}[h!]
	%\tiny 
	\scriptsize
	\centering
	\caption{QAPLIB Instances of Medium Size}
	\label{table:QAPLIB_II}
	\begin{tabular}{|l|cccccccc|}
		\hline %\toprule    % from the data file     
		%%Medium_Result_May5  Medium_Result_ADMM_May5  
		%%Medium_Result_tho30  Medium_Result_ADMM_May_tho30
		\rule{-3pt}{8pt}   
	\textbf{problem}&\textbf{opt}& \textbf{lbd} & \textbf{ubd} & 
		\textbf{rel.gap} & \textbf{rel.gap$^{\ADMM}$} & \textbf{iter} & 
		\textbf{iter$^{\ADMM}$} & \textbf{time(sec)}  \\
		\hline
chr22a & 6156 & 6156 & 6156 & \textbf{0} & 0 & 11500 & 14200 & 373.12 \\  
chr22b & 6194 & 6190 & 6194 & 0.06 & 0 & 13500 & 11500 & 467.28 \\  
chr25a & 3796 & 3796 & 3796 & \textbf{0} & 0 & 7000 & 6100 & 362.11 \\  
esc32a & 130 & 104 & 160 & \textbf{42.26} & 106.49 & 25000 & 14000 & 3956.78 \\  
esc32b & 168 & 132 & 216 & \textbf{48.14} & 95.87 & 700 & 10900 & 108.35 \\  
esc32c & 642 & 616 & 652 & \textbf{5.67} & 20.92 & 3000 & 2700 & 474.61 \\  
esc32d & 200 & 192 & 210 & \textbf{8.93} & 44.94 & 700 & 3400 & 109.00 \\  
esc32e & 2 & 2 & 24 & 162.96 & 147.37 & 600 & 8300 & 91.08 \\  
esc32g & 6 & 6 & 8 & \textbf{26.67} & 121.21 & 300 & 2000 & 48.71 \\  
esc32h & 438 & 426 & 456 & \textbf{6.80} & 26.68 & 9800 & 12000 & 1483.00 \\  
kra30a & 88900 & 86838 & 96230 & \textbf{10.26} & 14.54 & 6500 & 8500 & 784.84 \\  
kra30b & 91420 & 87858 & 101640 & \textbf{14.55} & 28.52 & 3600 & 12900 & 428.83 \\  
kra32 & 88700 & 85776 & 93950 & \textbf{9.10} & 34.43 & 3000 & 9100 & 460.33 \\  
nug21 & 2438 & 2382 & 2682 & \textbf{11.85} & 17.12 & 6000 & 19300 & 150.78 \\  
nug22 & 3596 & 3530 & 3678 & \textbf{4.11} & 16.79 & 6800 & 12100 & 210.12 \\  
nug24 & 3488 & 3402 & 3818 & \textbf{11.52} & 17.78 & 3700 & 11800 & 160.69 \\  
nug25 & 3744 & 3626 & 4024 & \textbf{10.40} & 19.06 & 7700 & 15900 & 395.78 \\  
nug27 & 5234 & 5130 & 5502 & \textbf{7.00} & 11.64 & 7000 & 12700 & 508.69 \\  
nug28 & 5166 & 5026 & 5674 & \textbf{12.11} & 17.14 & 5300 & 12300 & 455.90 \\  
nug30 & 6124 & 5950 & 6610 & \textbf{10.51} & 15.76 & 6900 & 12900 & 799.43 \\  
ste36a & 9526 & 9260 & 9980 & \textbf{7.48} & 39.68 & 14600 & 26700 & 4445.23 \\  
ste36b & 15852 & 15668 & 16058 & \textbf{2.46} & 84.83 & 40000 & 38500 & 11195.79 \\  
ste36c & 8239110 & 8134838 & 8387978 & \textbf{3.06} & 37.61 & 16800 & 40000 & 4036.27 \\  
tai25a & 1167256 & 1096658 & 1279534 & \textbf{15.39} & 20.55 & 1700 & 2300 & 71.73 \\  
tai30a & 1818146 & 1706872 & 1987862 & \textbf{15.21} & 15.21 & 3300 & 3700 & 319.41 \\  
tai35a* & 2422002 & 2216648 & 2598992 & \textbf{15.88} & 20.94 & 1800 & 3300 & 379.75 \\  
tai40a* & 3139370 & 2843314 & 3461270 & \textbf{19.60} & 22.87 & 2500 & 4700 & 1016.85 \\  
tho30 & 149936 & 143576 & 166336 & \textbf{14.69} & 23.62 & 5000 & 17900 & 582.65 \\  
tho40* & 240516 & 226522 & 258158 & \textbf{13.05} & 21.71 & 6200 & 21200 & 2323.48 \\ 
		\hline   % \bottomrule
	\end{tabular}
\end{table}

Columns \textbf{rel.gap} and \textbf{rel.gap$^\ADMMp$} in \Cref{table:QAPLIB_II} show that \rPRSM produces competitive relative gaps compared to \ADMMp.
In particular, 27 out of 29 instances are solved with relative gaps just 
as good as the ones obtained by \ADMM and these instances are marked with 
boldface in \Cref{table:QAPLIB_II}.
We have found provably optimal solutions for instances \emph{chr22a} and \emph{chr25a}.
We also observe from columns \textbf{iter} and \textbf{iter$^\ADMMp$} in 
\Cref{table:QAPLIB_II} that \rPRSM gives reduction in number of iterations 
in many instances;
24 out of 29 instances use fewer number of iterations using \rPRSM 
compared to \ADMMp.

For \rPRSM alone we observe that most of the instances show good bounds 
with reasonable amount of time. 
\rPRSM produces strong lower bounds on these instances, mostly 
within $5$ percent of the optimum.

\subsubsection{Large Size}

\Cref{table:QAPLIB_III} contains results on 9 QAPLIB instances with sizes $n \in \{41,\ldots, 64\}$.
We again make similar observations made in \Cref{sec:QAPLIB-Small}.
\begin{table}[h!]
	%\tiny 
	\scriptsize
	\centering
	\caption{QAPLIB Instances of Large Size}
	\label{table:QAPLIB_III}
	\begin{tabular}{|l|cccccccc|}
		\hline %\toprule 
		\rule{-3pt}{8pt}
		\rule{-3pt}{8pt}   
	    \textbf{problem}&\textbf{opt}& \textbf{lbd} & \textbf{ubd} & 
		\textbf{rel.gap} & \textbf{rel.gap$^{\ADMM}$} & \textbf{iter} & 
		\textbf{iter$^{\ADMM}$} & \textbf{time(sec)}  \\
		\hline
esc64a & 116 & 98 & 222 & \textbf{77.26} & 81.68 & 500 & 1400 & 6595.62 \\  
sko42* & 15812 & 15336 & 16394 & \textbf{6.67} & 17.61 & 5800 & 18200 & 4249.87 \\  
sko49* & 23386 & 22654 & 24268 & \textbf{6.88} & 17.41 & 7900 & 17300 & 14234.86 \\  
sko56* & 34458 & 33390 & 36638 & \textbf{9.28} & 15.13 & 5100 & 20600 & 20533.41 \\  
sko64* & 48498 & 47022 & 50684 & \textbf{7.50} & 15.37 & 6500 & 20900 & 66648.80 \\  
tai50a* & 4938796 & 4390982 & 5421576 & \textbf{21.01} & 25.79 & 2300 & 5400 & 4580.58 \\  
tai60a* & 7205962 & 6326350 & 7920830 & \textbf{22.38} & 25.60 & 3300 & 7400 & 23471.83 \\  
tai64c & 1855928 & 1811348 & 1887500 & \textbf{4.12} & 36.50 & 1200 & 2800 & 11054.54 \\  
wil50* & 48816 & 48126 & 50712 & \textbf{5.23} & 8.89 & 4700 & 15300 & 6133.16 \\  
		\hline %\bottomrule
	\end{tabular}
\end{table}
We observe that \rPRSM outputs better relative gaps than \ADMM on all 
these instances and this is due to the random perturbation approach 
presented in \Cref{sec:upper_bound}.
We also obtain reduction on the number of iterations. It indicates that 
our strategies taken on $R$ and $Z$ updates in \rPRSM help the iterates 
converges faster than \ADMMp.

\subsection{Comparisons to Other Methods}
\label{sec:experment2}
In this section we make comparisons with results from
two recent papers that engage 
\QAP lower and upper bounds via relaxation.\footnote{
	For more comparisons, see e.g.,~\cite[Table 4.1,~Table 
	4.2]{OliveiraWolkXu:15} to view a complete list of lower bounds using 
	bundle method presented in \cite{rendlsotirov:06}.}

\paragraph{Comparison to C-SDP(\!\!\!\cite{MR3772051})}
Here we compare our numerical result with the results presented by 
Ferreira 
et al.~\cite{MR3772051}.
Briefly, Ferreira et al. \cite{MR3772051} propose a semidefinite 
relaxation based algorithm C-SDP.
The algorithm applies to relatively sparse data and hence their results
are presented for chr and esc families in QAPLIB.
Figure \ref{fig:FKScomparison} below illustrates the relative gaps
arising from \rPRSM and C-SDP.
The numerics used in Figure \ref{fig:FKScomparison} can be found in 
\cite[Table 3-4]{MR3772051}.
\begin{figure}[h!]
	\centering   % WHY IT NEVER GETS CLOSE TO THE TEXT????  % fake plot
	\includegraphics[width=15.2cm, height=3.8cm]{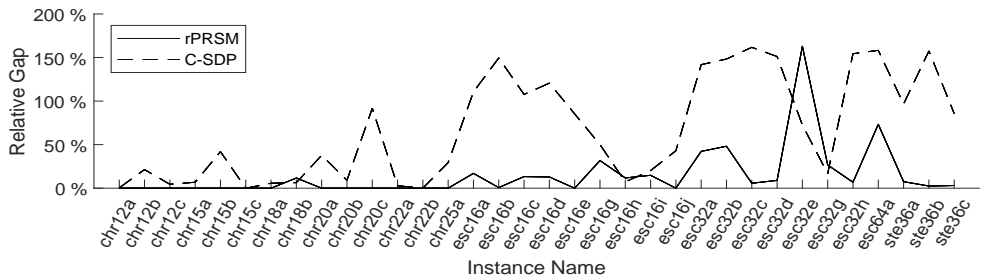}
	\vspace{-12pt}
	\caption{Relative Gap for \rPRSM and C-SDP}
	\label{fig:FKScomparison}
\end{figure}
The horizontal axis indicates the instance name on QAPLIB whereas the 
vertical axis indicates the relative gap\footnote{We selected the best 
result given in \cite[Table3, Table 4]{MR3772051} for different 
parameters. We point out that \cite{MR3772051} used a different formula 
for the gap computation. In this paper, we recomputed the relative gaps 
using \eqref{def:rel_gap} for a proper comparison.
	\cite{MR3772051} used similar approach for upper bounds as in our 
	paper, that is, the projection onto permutation matrices using 
	\cite{MR0054920,Birk:46}.}.
Figure \ref{fig:FKScomparison} illustrates that \rPRSM yields much 
stronger relative gaps than C-SDP.

\paragraph{Comparison to F2-RLT2-DA(\!\!\!\cite{MR4037827})}

Date and Nagi \cite{MR4037827} propose F2-RLT2-DA, a linearization 
technique-based parallel algorithm (GPU-based) for obtaining  lower bounds 
via Lagrangian relaxation.

\begin{figure}[h]
	\centering  
	\subfigure[Lower Bound Gap]{\includegraphics[width=7.5cm, 
	height=3.6cm]{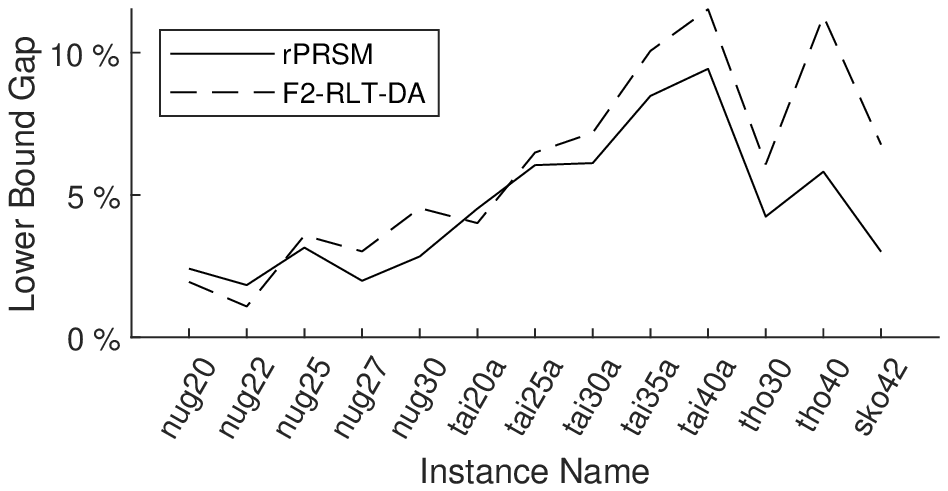}} %\label{fig:lbd_DN}
	\subfigure[Running Time]{\includegraphics[width=7.5cm, 
	height=3.6cm]{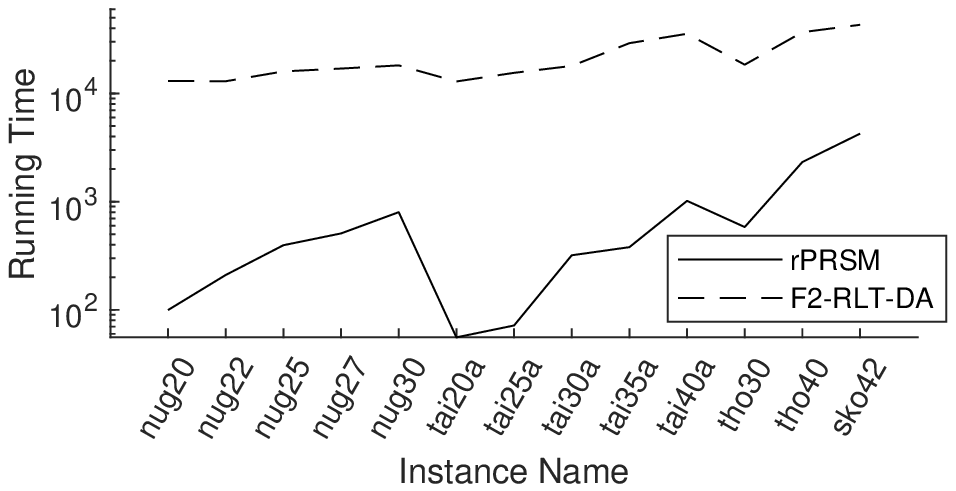}}  %\label{fig:time_DN}
	\vspace{-12pt}
	\caption{Numerical Comparison for \rPRSM and F2-RLT2-DA}
	\label{fig:DNcomparison}
\end{figure}

Figure \ref{fig:DNcomparison}(a) illustrates the comparisons on lower 
bound gap \footnote{We compute the lower bound gap by $100*(p^*-l)/p^*\%$, 
where $p^*$ is the best known feasible value to \QAP and $l$ is the lower 
bound.}
using \rPRSM and F2-RLT2-DA. 
It shows that both \rPRSM and F2-RLT2-DA output competitive lower bounds 
to the best known feasible values for \QAPp.
Figure \ref{fig:DNcomparison}(b) illustrates the comparisons on the 
running time\footnote{The running time for F2-RLT2-DA is 
obtained by using the average time per iteration presented in \cite{MR4037827} multiplied by 2000 as F2-RLT2-DA 
runs the algorithm for 2000 iterations. The running time for \rPRSM is 
drawn from \Cref{table:QAPLIB_I,table:QAPLIB_II,table:QAPLIB_III}.}
in seconds using \rPRSM and F2-RLT2-DA.
We observe that the running time of F2-RLT2-DA is much longer than the 
running time of \rPRSMp; F2-RLT2-DA requires at least 10 times longer than 
\rPRSMp. Furthermore, from Figure \ref{fig:DNcomparison} we observe that 
even 
though the two 
methods give similar lower bounds to \QAPp, \rPRSM is less time-consuming 
even considering the differences in the hardware\footnote{F2-RLT2-DA was 
coded in C++ and CUDA C programming languages and deployed on the Blue 
Waters Supercomputing facility at the University
	of Illinois at Urbana-Champaign. Each processing element consists
	of an AMD Interlagos model 6276 CPU with eight
	cores, 2.3 GHz clock speed, and 32 GB memory connected
	to an NVIDIA GK110 ``Kepler'' K20X GPU with
	2,688 processor cores and 6 GB memory.}.

\section{Conclusion}
\label{sect:concl}

In this paper
we re-examin the strength of using a splitting method for solving 
the facially reduced \SDP relaxation for the \QAP with nonnegativity
constraints added, that is, 
the splitting of constraints into two subproblems that are challenging
to solve when used together.
In addition, we provide a straightforward derivation of facial
reduction and the gangster constraints via a direct lifting.

We use a strengthened model and algorithm, i.e.,~we
incorporate a redundant trace constraint to the model that 
is not redundant in the subproblem from the splitting.   
We also exploit the set of dual optimal multipliers and provide 
customized dual updates in the algorithm that allow for a proof of
optimality for the original \QAPp. This allowed for a
new strategy for strengthening the upper and lower bounds.

\printindex
\addcontentsline{toc}{section}{Index}
\label{ind:index}

\bibliographystyle{plain}
%\bibliography{.master,.edm,.psd,.qap,.bjorBOOK, xinxin, .haesol}
\bibliography{.master,.psd,.qap}
\addcontentsline{toc}{section}{Bibliography}
%\bibliography{.master,.psd,.qap, xinxin, .haesol}

\end{document}